\documentclass[11pt,reqno]{amsart}

\usepackage[english]{babel}
\usepackage[utf8]{inputenc}
\usepackage[T1]{fontenc}
\usepackage{amsmath,amsfonts,amssymb,amsthm}
\usepackage{array}
\usepackage{mathdots}
\usepackage{mathtools} %
\usepackage{hyperref}
\usepackage{enumitem}
\usepackage[capitalise]{cleveref} 
\usepackage{xspace} %

\selectfont
\makeatletter 
\def\l@subsection{\@tocline{2}{0pt}{2pc}{6pc}{}}
\makeatother

\DeclareUnicodeCharacter{2264}{\ensuremath{\leqslant}}
\DeclareUnicodeCharacter{2265}{\ensuremath{\geqslant}}
\DeclareUnicodeCharacter{2260}{\ensuremath{\neq}}

\newtheorem{mainthm}{Theorem}

\theoremstyle{plain}
\newtheorem{lem}{Lemma}[section]
\newtheorem{thm}[lem]{Theorem}
\newtheorem{cor}[lem]{Corollary}
\newtheorem{prop}[lem]{Proposition}

\theoremstyle{definition}
\newtheorem{algo}[lem]{Algorithm}
\newtheorem{defi}[lem]{Definition}
\newtheorem{rem}[lem]{Remark}

\newtheorem{mainconj}[mainthm]{Conjecture}

\numberwithin{figure}{section}
\numberwithin{table}{section}
\numberwithin{equation}{section}

\crefname{section}{Section}{Sections} 
\crefname{table}{Table}{Tables} 

\crefname{thm}{Theorem}{Theorems} 
\crefname{mainthm}{Theorem}{Theorems} 
\crefname{prop}{Proposition}{Propositions} 
\crefname{cor}{Corollary}{Corollaries} 
\crefname{lem}{Lemma}{Lemmas} 
\crefname{ex}{Example}{Examples} 
\crefname{mainconj}{Conjecture}{Conjectures} 
\crefname{algo}{Algorithm}{Algorithms} 
\crefname{defi}{Definition}{Definitions} 

\AddToHook{env/prop/begin}{\crefalias{lem}{prop}}
\AddToHook{env/thm/begin}{\crefalias{lem}{thm}}
\AddToHook{env/cor/begin}{\crefalias{lem}{cor}}
\AddToHook{env/algo/begin}{\crefalias{lem}{algo}}
\AddToHook{env/defi/begin}{\crefalias{lem}{defi}}
\AddToHook{env/rem/begin}{\crefalias{lem}{rem}}
\AddToHook{env/ex/begin}{\crefalias{lem}{ex}}
\AddToHook{env/exemples/begin}{\crefalias{lem}{exemples}}
\AddToHook{env/mainconj/begin}{\crefalias{mainthm}{mainconj}}

\newcommand\andname{and}

\setlist{topsep= 2pt plus 1pt minus 1pt,itemsep= 1pt plus 1pt minus 1pt}
\setlist[itemize]{leftmargin=1.5\parindent}
\setlist[itemize,1]{label=$\bullet$}%
\setlist[itemize,2]{label=$\circ$}%
\setlist[enumerate,1]{leftmargin=2em,label=\arabic*), ref=\arabic*,itemjoin*=\quad}

\makeatletter
\newcommand{\cfp}{\,\mathrel{\cfracplus@}\,}%
\newcommand{\cfracplus@}{%
  \sbox\z@{$\dfrac{1}{1}$}%
  \sbox\tw@{$+$}%
  \raisebox{\dimexpr\dp\tw@-\dp\z@\relax}{$+$}%
}
\newcommand{\cfd}{\mathord{\cfracdots@}}
\newcommand{\cfracdots@}{%
  \sbox\z@{$\dfrac{1}{1}$}%
  \sbox\tw@{$+$}%
  \raisebox{\dimexpr\dp\tw@-\dp\z@\relax}{$\cdots$}%
}
\makeatother

\newcommand{\Q}{\mathbb{Q}}
\newcommand{\Z}{\mathbb{Z}}

\DeclareMathOperator{\val}{val}
\DeclareMathOperator{\ord}{ord}
\DeclareMathOperator{\sgn}{sgn}

\newcommand{\cI}{\mathcal{I}}
\newcommand{\cP}{\mathcal{P}}
\newcommand{\cH}{\mathcal{H}}

\newcommand\de{\delta}
\newcommand\De{\Delta}

\newcommand\ka{\kappa}

\newcommand\nabc{{\tt NextABC}\xspace}
\let\fr=\frac

\begin{document}

\title[Hankel determinants of the $q$-metallic numbers]{Hankel determinants \\
for the $(n+2)$-shifted $q$-metallic numbers}

\author{Guo-Niu Han}
\address{%
Guo-Niu Han,
Université de Strasbourg, CNRS, IRMA, UMR~7501, Strasbourg, France}
\email{guoniu.han@unistra.fr}

\author{Emmanuel Pedon}
\address{%
Emmanuel Pedon,
Universit\'e de Reims Champagne-Ardenne,
CNRS, LMR, UMR~9008, Reims, France}
\email{emmanuel.pedon@univ-reims.fr}

\subjclass[2020]{Primary: 05A30, 11A55, 11C20. Secondary: 11B37.}
\keywords{Shifted Hankel determinants, Hankel continued fractions, metallic numbers,
$q$-analogues, quadratic irrational series}
\date{October 26, 2026}

\begin{abstract}
Let $\Phi_n(q)$ denote the $q$-deformation (in the sense of Morier-Genoud \& Ovsienko)  of the metallic number $\phi_n=\frac{n+\sqrt{n^2+4}}{2}$. This is a power series with integer coefficients. We prove that the Hankel
continued fraction of the $(n{+}2)$-shifted series $\Phi_n^{(n+2)}$ is
eventually periodic, and that the corresponding Hankel determinant sequence
$\Delta^{(n+2)}$ is $2n(n+1)$-periodic if $n$ is even, $2n(n+1)$-antiperiodic if $n$ is odd, and that it takes only the values
$-2,-1,0,1,2$. This confirms part of a former conjecture of ours and complements similar results we obtained for all $\ell$-shifts of $\Phi_n$ with $0≤\ell≤n+1$. In addition, we formulate a new conjecture of quasi-polynomiality for higher shifts $\ell≥n+3$.

The determinant values arise as telescoping products of the partial numerators of
the Hankel continued fraction. Most of the technically involved and lengthy
verifications are carried out by a symbolic computation program, whose code is
publicly available on GitHub for independent verification. 
\end{abstract}

\maketitle

\thispagestyle{empty}

\section{Introduction}
\label{sec:intro}

The so-called \emph{metallic numbers} (or \emph{metallic ratios}, or \emph{metallic means}) are the quadratic irrationals

\begin{equation}\label{defphin}
\phi_n:=\frac{n+\sqrt{n^2+4}}{2}
 =n+\fr{1}{n}\cfp\fr{1}{n}\cfp\fr{1}{n}\cfp\cfd
 =n+\left(\fr 1{n}\cfp\right)^* \qquad (n≥1).
\end{equation} 
In particular, $\phi_1=\frac{1+\sqrt{5}}{2}$ is the \emph{golden ratio}. Here and thereafter, we use the following classical notation for infinite continued fractions:
\begin{equation*}
  \frac{\alpha_0}{\beta_0}\cfp\frac{\alpha_1}{\beta_1}\cfp
  \frac{\alpha_2}{\beta_2}\cfp\cdots
  \quad:=\quad\cfrac{\alpha_0}{\beta_0+\cfrac{\alpha_1}{\beta_1+
    \cfrac{\alpha_2}{\beta_2+\cdots}}}
\end{equation*}
and indicate periodic data between  parentheses and with a superscript $*$.

Thanks to S.~Morier-Genoud and V.~Ovsienko's fascinating theory of $q$-deformed real numbers (see \cite{MGO20,MGO22} for the foundations, \cite{MGO} for a survey of its numerous occurrences in various branches of mathematics and \cite{Pedon,Etingof,EO} for recent developments on analytical aspects of $q$-numbers), one can attach to any metallic ratio $\phi_n$ an algebraic continued fraction $[\phi_n]_q$ called the \emph{$q$-deformation} of $\phi_n$.

By definition, $[\phi_n]_q$ is the following quantisation of the continued fraction \eqref{defphin}:
\begin{equation}\label{defPhin}
[\phi_n]_q:=
[n]_q+\left(\fr{q^n}{[n]_{q^{-1}}}\cfp\fr{q^{-n}}{[n]_q}\cfp\right)^*
\end{equation}
where the polynomial $[n]_q$ stands for the classical Euler-Gauss $q$-deformation of the integer $n$:
\begin{equation*}
[n]_q:=1+q+q^2+\cdots+q^{n-1}.
\end{equation*}
As expected when speaking of quantisation, letting $q\to 1$ in \eqref{defPhin} gives back expression \eqref{defphin} of $\phi_n$.
These algebraic continued fractions $[\phi_n]_q$ are called \emph{the $q$-metallic numbers}. 

A key fact of the $q$-deformation theory is that there exists a sequence of \emph{integers} $(\ka_l(\phi_n))_{l≥0}$, such that
\begin{equation*}
[\phi_n]_q=\sum_{l=0}^{+\infty}\ka_l(\phi_n) q^l,
\end{equation*}
i.e. $[\phi_n]_q$ expands about $q=0$ as a formal power series in $\Z[[q]]$. Actually, this is a true power series expansion (with non-zero radius of convergence) when $q$ is a complex variable, see \cite{Etingof,EO}.
To view $[\phi_n]_q$ as a function of $q$ we will use the following notation
\begin{equation}\label{Phin-SE}
\Phi_n(q)=[\phi_n]_q=\sum_{l=0}^{+\infty}\ka_l(\phi_n) q^l
\end{equation}
and still call the function $\Phi_n=\Phi_n(q)$ the $n$-th $q$-metallic number.

The $q$-metallic numbers are studied in depth in the articles \cite{OP25}, \cite{HP}, \cite{Pedon}. They reveal strong connections between the sequence of coefficients $\bigl(\ka_l(\phi_n)\bigr)_{l≥0}$ in \eqref{Phin-SE} and three highly famous sequences of the combinatorial world:  the {Catalan numbers}, the {Motzkin numbers} and the enumeration of {secondary structures of RNA molecules}. These connections concern mainly the analytical properties of the sequences and their Hankel determinants.

About the latter, recall that for a power series $F(q)=\sum_{i=0}^{+\infty}f_iq^i$,
or equivalently a sequence of numbers $F=(f_i)_{i≥0}$, the \emph{Hankel determinants of $F$} are the doubly indexed family 
\begin{equation}
\label{HankDet}
\De^{(\ell)}_j(F):=\det\left((f_{\ell+a+b})_{a,b=0}^{j-1}\right)
=\left|
\begin{matrix}
{f}_\ell & {f}_{\ell+1} & \cdots & {f}_{\ell+j-1} \\[4pt]
{f}_{\ell+1} &{f}_{\ell+2} & \cdots & {f}_{\ell+j} \\[4pt]
\vdots & \vdots && \vdots \\
 {f}_{\ell+j-1} & {f}_{\ell+j} & \cdots & {f}_{\ell+2j-2}
\end{matrix}
\right|
,\qquad \ell,j≥0,
\end{equation}
with  convention 
\begin{equation}\label{conv0}
\De^{(\ell)}_0(F):=1.
\end{equation} 
The determinant $\De^{(\ell)}_j(F)$ is  of size $j\times j$, and we call the number $\ell$ the \emph{shift} of the determinant, since it is clear that
\begin{equation}\label{Delo}
\De^{(\ell)}_j\bigl((f_i)_{i≥0}\bigr)=\De_j^{(0)}\bigl((f_{i+\ell})_{i≥0}\bigr).
\end{equation}
In other words,
\begin{equation*}\label{defFl}
\De^{(\ell)}_j(F)=\De_j^{(0)}(F^{(\ell)})
\end{equation*}
where $F^{(\ell)}(q)$ is the power series
obtained by deleting the first $\ell$ Taylor coefficients of $F(q)$ and dividing by
$q^\ell$:

\begin{equation}\label{defshift}
F^{(\ell)}(q):=
\frac 1{q^\ell}\biggl(F(q)-\sum_{i=0}^{\ell-1}f_iq^i\biggr)
=\sum_{i=0}^{+\infty}f_{i+\ell}q^{i}.
\end{equation}
When $\ell=0$ the numbers $\De^{(0)}_j(F)$ will be called \emph{ordinary Hankel determinants} of $F$. Often, we will write simply $\De^{(\ell)}(F)$ to denote the entire sequence $(\De^{(\ell)}_j(F))_{j≥0}$.

Also, we write, for short,
\[
  \De^{(\ell)}_j:=\De^{(\ell)}_j(\Phi_n)=\De_j^{(0)}(\Phi_n^{(\ell)}).
\]
It turns out that sequences $\De^{(\ell)}$, with $0≤\ell≤n+1$, enjoy many striking properties, among which we emphasize the following (see \cite{OP25} for case $n=1,2$ and \cite{HP} for general $n≥2$): 
\begin{enumerate}[label=\textbf{(\Roman*)},leftmargin=7ex]
\item \textbf{Periodicity.} The sequences $\Delta^{(\ell)}$, with $0≤\ell≤n+1$, are $2n(n+1)$-periodic when $n$ is even and $2n(n+1)$-antiperiodic when $n$ is odd.
\item \textbf{Values.} They take the values $-1,0,1$ only. The distribution
of these values in $\De^{(0)}$ is governed by the symmetry
\begin{equation*}%
  \De_j=(-1)^{n(n+1)/2}\,\De_{(2n+1)(n+1)-j}\qquad
  \text{for }0≤ j≤(2n+1)(n+1),
\end{equation*}
and the shifted sequences are obtained from $\De^{(0)}$ by the contiguity
relation
$$\Delta^{(\ell+1)}_{j}=(-1)^{j+\frac{n(n+2\ell-1)}{2}}\Delta^{(\ell)}_{j+n+1}
\quad\text{for all } j≥0.$$
\item \textbf{Recurrence.} The sequences $\Delta^{(\ell)}$, $0≤\ell≤ n+1$, all satisfy the same \emph{Somos--Gale--Robinson recurrence}:
\begin{equation}\label{GR}
\De^{(\ell)}_{j+2n+2}\,\De^{(\ell)}_j = \De^{(\ell)}_{j+2n+1}\,\De^{(\ell)}_{j+1} - \bigl(\De^{(\ell)}_{j+n+1}\bigr)^2
\quad\text{for all }  j≥0.
\end{equation}
Hence they form periodic discrete integrable dynamical systems\footnote{We refer to \cite{HP} for some information on Somos-Gale-Robinson recurrences, which will not be developed here.}.
\end{enumerate}

All these results are based on the expansion of the functions $\Phi_n^{(\ell)}$ (for $0≤\ell≤n+1$) as a special type of continued fraction introduced in \cite{Han16}, called \emph{Hankel continued fraction}, whose partial numerators determine all the Hankel
determinants of the series; see Section~\ref{sec:hfrac} for the precise
definition.

The main purpose of this paper is to treat the case $\ell=n+2$, technically much more difficult than for the lower shifts $0≤\ell≤n+1$, and to make a precise conjecture as concerns the shifts $\ell≥n+3$, which behave in a
completely different way  (see below). 

Concerning case $\ell=n+2$, we will mainly prove the analogues of properties (I) \& (II) above by following the same path, whose first step is:

\begin{mainthm}\label{mainthmA}
Let $n≥1$ be an integer. The Hankel continued fraction of the shifted $q$-metallic number $\Phi_n^{(n+2)}$ is eventually periodic:
\begin{equation*}
\Phi_n^{(n+2)}(q)=\cI_n(q)\cfp\Bigl(\cP_n(q)\Bigr)^*.
\end{equation*}
More precisely,  the pre-periodic block $\cI_n(q)$ consists of one single term, given by
\begin{equation*}
 \cI_1(q)=\frac{-1}{1+2q},
 \qquad
 \cI_n(q)=\frac{q^{n-2}}{(q^n-Q)/(q-1)}\quad(n≥2),
\end{equation*}
with
\begin{equation*}
Q:=q^2-q+1,
\end{equation*}
and the periodic block $\cP_n(q)$ consists of $p_n$ terms, where
\begin{equation*}%
p_n:=
\begin{dcases*}
2&if $n=1$,\\
9&if $n=2$,\\
9n-13&if $n≥3$ is odd,\\
9n-10&if $n≥4$ is even.
\end{dcases*}
\end{equation*}
\end{mainthm}

Note that the period $p_n$ need not be the minimal one (although it is always the case in our numerical examples). 

The complete Hankel continued fraction of $\Phi_n^{(n+2)}$ produced by \cref{mainthmA} being too complicated to be written in the general case, let us provide a few examples.

Case $n=1$, $p_n=2$:
\[
\Phi_1^{(3)}(q)=
\frac{-1}{1+2q}\cfp
\left(
\frac{-q^4}{1+q-q^2+2q^3}\cfp
\frac{-q^4}{1+2q}
\right)^*
\]

Case $n=2$, $p_n=9$:
\[
\begin{aligned}
\Phi_2^{(4)}(q)={}&\frac{1}{1}\cfp\left( 
\frac{2q^2}{1+\frac12q}\cfp
\frac{\frac14q^2}{1+\frac32q}\cfp
\frac{4q^2}{1-\frac32q}\cfp
\frac{-\frac14q^2}{1-\frac32q}\right.\\
&\left.\cfp
\frac{4q^2}{1+\frac32q}\cfp
\frac{\frac14q^2}{1+\frac12q}\cfp
\frac{2q^2}{1}\cfp
\frac{q^5}{1+2q^2-q^3}\cfp
\frac{q^5}{1}\right)^*.
\end{aligned}
\]

Case $n=3$, $p_n=14$:
\[
\begin{aligned}
\Phi_3^{(5)}(q)={}&\frac{q}{1+q^2}\cfp\left( 
\frac{2q^3}{1+\frac12q}\cfp
\frac{\frac14q^2}{1-\frac12q}\cfp
\frac{2q^2}{1}\cfp
\frac{-q^2}{1+q}\cfp
\frac{-q^3}{1-q+2q^2}\right.\\
&\left.\cfp
\frac{-q^5}{1+q+q^2+2q^3}\cfp
\frac{-q^5}{1-q+2q^2}\cfp
\frac{-q^3}{1+q}\cfp
\frac{-q^2}{1}\right.\\
&\left.\cfp
\frac{2q^2}{1-\frac12q}\cfp
\frac{\frac14q^2}{1+\frac12q}\cfp
\frac{2q^3}{1+q^2}\cfp
\frac{q^7}{1+q^2+2q^3-q^4}\cfp
\frac{q^7}{1+q^2}\right)^*.
\end{aligned}
\]

Case $n=4$, $p_n=26$:
\[
\begin{aligned}
\Phi_4^{(6)}(q)={}&\frac{q^2}{1+q^2+q^3}\cfp\left( 
\frac{2q^4}{1+\frac12q}\cfp
\frac{\frac14q^2}{1-\frac12q}\cfp
\frac{2q^3}{1+q^2}\cfp
\frac{-q^3}{1+q}\cfp
\frac{-q^3}{1-q+q^2}\right.\\
&\left.\cfp
\frac{q^3}{1}\cfp
\frac{q^2}{1}\cfp
\frac{-q^2}{1+q}\cfp
\frac{2q^2}{1-\frac12q}\cfp
\frac{-\frac14q^2}{1+\frac32q}\cfp
\frac{4q^2}{1-2q}\cfp
\frac{\frac12q^3}{1+q+q^2}\right.\\
&\left.\cfp
\frac{\frac12q^3}{1-2q}\cfp
\frac{4q^2}{1+\frac32q}\cfp
\frac{-\frac14q^2}{1-\frac12q}\cfp
\frac{2q^2}{1+q}\cfp
\frac{-q^2}{1}\cfp
\frac{q^2}{1}\right.\\
&\left.\cfp
\frac{q^3}{1-q+q^2}\cfp
\frac{-q^3}{1+q}\cfp
\frac{-q^3}{1+q^2}\cfp
\frac{2q^3}{1-\frac12q}\cfp
\frac{\frac14q^2}{1+\frac12q}\right.\\
&\left.\cfp
\frac{2q^4}{1+q^2+q^3}\cfp
\frac{q^9}{1+q^2+q^3+2q^4-q^5}\cfp
\frac{q^9}{1+q^2+q^3}\right)^*.
\end{aligned}
\]

Our second main result concerns the $(n+2)$-shifted Hankel sequence of the $q$-metallic number $\Phi_n(q)$, i.e., the ordinary Hankel sequence of the $(n+2)$-shifted $q$-metallic number $\Phi_n^{(n+2)}(q)$.

\begin{mainthm}\label{mainthmB}
Let $n≥1$ be an integer.
The sequence $\De^{(n+2)}=\bigl(\De^{(n+2)}_j\bigr)_{j≥0}$ of $(n+2)$-shifted Hankel determinants of the $q$-metallic number $\Phi_n(q)$ satisfies the following properties:
\begin{enumerate}
\item \emph{Periodicity.} The sequence is $2n(n+1)$-periodic when $n$ is even and
$2n(n+1)$-antiperiodic (hence $4n(n+1)$-periodic) when $n$ is
odd;
\item \emph{Values.} It consists of the values $-2,-1,0,1,2$ only (except if $n=1$: it can only
take the values $-1,0,1$);
\end{enumerate}
\end{mainthm}

Thus, the (anti)periodicity property is exactly the same as for the lower shifts $0≤\ell≤n+1$, see (I)~above. On the contrary the set of values is strictly bigger (except for $n=1$) but still remarkable. For instance, one period of $\De^{(n+2)}$ for $n=2$, and one antiperiod for
$n=3$, read respectively
\begin{gather*}
  1,1,-2,-1,2,-1,-2,1,1,0,0,0,\\
  1,0,-1,2,1,-1,1,0,-1,0,0,-1,0,1,-1,1,2,-1,0,1,0,0,0,0 .
\end{gather*}

Let us make two comments on \cref{mainthmB}. 
Firstly, it proves part~(1)
of \cite[Conjecture~E]{HP}; the case $n=1$ was already proved in
\cite[Theorem~1.6]{OP25}, but we treat here all cases $n≥1$. 
Secondly, we mention that we have sought for three-term  Somos--Gale--Robinson recurrences of type \eqref{GR}
\begin{align*}
  &\Delta_j\Delta_{j+k}
  =
  \alpha\Delta_{j+a}\Delta_{j+k-a}
  +
  \beta\Delta_{j+b}\Delta_{j+k-b},\\
&\text{with}\quad
  1≤ a<b≤ \frac{k}{2},
  \quad
  \alpha\beta\neq0
\end{align*}
which could be satisfied by the sequence $\De^{(n+2)}$, but we could not find any for $n=1,\ldots,5$ and $k≤40$. Nor did we find any four-term recurrence
\begin{align*}
  &
  \De_j\De_{j+k}=\alpha\,\De_{j+a}\De_{j+k-a}+\beta\,\De_{j+b}\De_{j+k-b}
  +\gamma\,\De_{j+c}\De_{j+k-c},\\
&\text{with}\quad 1≤ a<b<c≤\frac k2,\quad \alpha\beta\gamma\ne0,
\end{align*}
for $n=2,\dots,5$ and $6≤ k≤40$.

A Hankel determinant sequence which is at the same time periodic and confined to the fixed finite set $\{-1,0,1\}$ (or even
$\{-2,-1,0,1,2\}$) is a rather unusual phenomenon. We found only two examples in the literature: \cite{SX08} and \cite[Example~6]{Barry12}. In our case, the situation is all the more remarkable because a single further shift $\ell=n+3$ of the sequence already  completely changes everything.
Indeed,  we conjectured in \cite{HP} that all sequences $\De^{(\ell)}$  were unbounded if $\ell≥ n+3$. 
Computer experimentation now leads us to formulate a much more precise conjecture:

\begin{mainconj}\label{mainconj}
Let $n≥1$ be an integer. For any $\ell≥n+3$, the shifted Hankel
determinants $\De_j^{(\ell)}=\De_j^{(\ell)}(\Phi_n)$ are quasi-polynomials of period
\begin{equation*}
\rho_n=\begin{dcases*}
2n(n+1)&if $n$ is even,\\
4n(n+1)&if $n$ is odd.\end{dcases*}
\end{equation*}
More precisely, there exists a family $\bigl(Q_{\ell,r}\bigr)_{r=0}^{\rho_n-1}$ of polynomials in $\Q[X]$, taking integer values on the integers, such that
for all $0≤ r≤\rho_n-1$ and $m≥0$,
\begin{equation*}\label{eq:conj}
  \De_{r+m \rho_n}^{(\ell)}=Q_{\ell,r}(m).
\end{equation*}
\end{mainconj}

Note finally that, \emph{modulo a prime number}  the Hankel determinants sequences $\De^{(\ell)}(\Phi_n)$ are periodic for \emph{every $n≥1$} and \emph{every shift $\ell≥0$}: see \cite[Theorem~F]{HP}.

\medskip
\noindent\textbf{Strategy and organisation of the paper.}

To establish Theorem~\ref{mainthmA} we use the 
approach developed in \cite{Han16}. The main tool is an algorithm \texttt{NextABC}
which builds the Hankel continued fraction ($H$-fraction, for short) of a
quadratic power series one step at a time. 
This is explained in \cref{sec:setup}, together with the quadratic equation verified by the shifted power series $\Phi_n^{(n+2)}$ of the $q$-deformed metallic ratio (Proposition~\ref{prop:init}).

Section~\ref{sec:hfrac} is devoted to the construction of the $H$-fraction of
$\Phi_n^{(n+2)}$ and to the proof of Theorem~\ref{mainthmA}; this is the most
technical part of the paper. At each step $m$, the algorithm \texttt{NextABC}
produces a quintuple $(A_m,B_m,C_m;K_m,D_m)$ of polynomials, from which the
$m$-th partial numerator and denominator of the $H$-fraction are read off. It
turns out that the quintuple at index $p_n+2$ coincides with the quintuple at
index~$2$, and this gives the periodicity. But the law which governs the
quintuples is complicated, and not even uniform in $m$: for $n≥4$, the
quintuples belong either to two families, called \emph{passes}, made of
repeating blocks of nine formulas (one for each residue of $m$ modulo $9$),
or to three short lists of quintuples, called \emph{bridges}, which occur only
once per period. These explicit lists, the \emph{pass-tables}, are given in
the appendices, and we prove that they are correct. Because of the size of
the expressions, this proof uses \textsc{SageMath} programs, which certify the
formulas symbolically, uniformly in $n$ and $m$. The small cases $n=1,2,3$ are
simpler and are treated separately.

In Section~\ref{sec:Hankel} we prove Theorem~\ref{mainthmB} by applying
\cite[Theorem~2.1]{Han16} (recalled as Theorem~\ref{th:super2} below), which
expresses the Hankel determinants of a power series in terms of the partial
numerators of its $H$-fraction. These partial numerators are known from the
pass-tables, but some care is needed to show that all the nonzero
determinants belong to $\{-2,-1,1,2\}$ and to determine the sign of the
(anti)periodicity.

Finally, Appendices~\ref{sec:n1}--\ref{sec:bridge2} list the
explicit tables for the 2 passes and the 3 bridges that are used in \Cref{sec:hfrac} to prove the periodicity of the Hankel continued fraction of $\Phi_n^{(n+2)}$.

\medskip
\noindent\textbf{Note.} All the \textsc{SageMath} programs, their output logs,
and the pass-table data are available at
\begin{center}
\url{https://github.com/GuoniuHan/metallicplus}
\end{center}
We emphasise that these programs deliver rigorous symbolic proofs of our
general statements, valid for all $n$ and $j$ at once, rather than checks of
finitely many cases.

\section{Background material and setup}
\label{sec:setup}

\subsection{The Hankel continued fraction and the algorithm \texttt{NextABC}}
\label{sec:HF-NABC}

We first recall from \cite{Han16} the definition and the basic property of
Hankel continued fractions. Throughout, the base field is $\Q$.

\begin{defi}\label{def:super}
For each positive integer $\delta$, a \emph{super continued fraction associated
with $\delta$}, called \emph{super $\delta$-fraction} for short, is a continued
fraction of the form
\begin{equation}\label{eq:super}
F(q)=\fr{v_0\,q^{\sigma_0}}{1+q u_1(q)}\cfp\fr{-v_1\,q^{\sigma_0+\sigma_1+\de}}{1+qu_2(q)}\cfp
\fr{-v_2\,q^{\sigma_1+\sigma_2+\de}}{1+q u_3(q)}\cfp
\fr{-v_3\,q^{\sigma_2+\sigma_3+\de}}{1+q u_4(q)}\cfp\cfd
\end{equation}
where the $v_j\ne0$ are constants, the $\sigma_j$ are nonnegative integers, and the
$u_j(q)$ are polynomials of degree less than or equal to $\sigma_{j-1}+\delta-2$ (by
convention, $0$ has degree $-1$).
\end{defi}

When $\delta=1$ (resp.\ $\delta=2$) and all $\sigma_j=0$, the super $\delta$-fraction
\eqref{eq:super} is the classical $S$-fraction (resp.\ $J$-fraction). A super
$2$-fraction is called a \emph{Hankel continued fraction}, or \emph{$H$-fraction}
for short.

The following result will be the main tool used in this article.

\begin{thm}[{\cite[Theorem~2.1]{Han16}}]\label{th:super2}
\begin{enumerate}
\item Let $\delta$ be a positive integer. Each super $\delta$-fraction
defines a power series, and conversely, for each power series $F(q)$ the super
$\delta$-fraction expansion of $F(q)$ exists and is unique.
\item Let $F(q)$ be a power series whose $H$-fraction is given by
\eqref{eq:super} with $\delta=2$. Then the non-vanishing Hankel determinants of
$F(q)$ are exactly those indexed by $s_j:=\sigma_0+\sigma_1+\cdots+\sigma_{j-1}+j$ for every $j≥0$, and their values are given  by the formula
\begin{equation}\label{eq:HankelDet}
  \De_{s_j}\bigl(F(q)\bigr)=(-1)^{\epsilon_j}\,
  v_0^{\,s_j}v_1^{\,s_j-s_1}v_2^{\,s_j-s_2}\cdots v_{j-1}^{\,s_j-s_{j-1}},
\end{equation}
where $\epsilon_j=\sum_{i=0}^{j-1}\sigma_i(\sigma_i+1)/2$.
\end{enumerate}  
\end{thm}

Notice that the notion of $H$-fraction  is also connected to  \emph{Artin's regular continued fraction}, see \cite[§2.2]{HP}.

Because we are only interested in $H$-fractions, from now on we  fix $\de=2$ and  set the following synthetic notation:
\begin{equation}\label{eq:Hfracstd}
  \cH\bigl(q;(\sigma_j),(v_j),(u_j)\bigr):=
\fr{v_0\,q^{\sigma_0}}{1+q u_1(q)}\cfp\fr{-v_1\,q^{\sigma_0+\sigma_1+2}}{1+qu_2(q)}\cfp
\fr{-v_2\,q^{\sigma_1+\sigma_2+2}}{1+q u_3(q)}\cfp
\cfd
\end{equation}
We recall from \cite[Algorithm~3.1]{Han16} (see also \cite[§3.1]{HP}) how the $H$-fraction can be constructed algorithmically for  power series 
 $F(q)\in\Q[[q]]$ satisfying a quadratic equation
\begin{equation}\label{Fquad}
  A(q)+B(q)F(q)+C(q)F(q)^2=0
\end{equation}
with $A,B,C\in \Q[q]$ such that
\begin{equation}
A≠0,\quad B(0)=1,\quad C≠0,\quad C(0)=0.\label{condABC}
\end{equation}

\begin{algo}\label{alg}{} [Algorithm \nabc for $H$-fractions] 
\begin{description}
\item [Prototype] $(A^*, B^*, C^*; k, a, D)=\nabc(A,B,C)$.
\item [Input]  $A(q),B(q),C(q)\in\Q[q]$ three polynomials satisfying conditions \eqref{condABC}.
\item [Output] an integer $k≥0$,\\
a nonzero scalar $a\in\Q$,\\
 a polynomial $D\in\Q[q]$ of degree $≤k+1$ such that $D(0)=1$,\\
three polynomials $A^*, B^*, C^*\in\Q[q]$ satisfying  \eqref{condABC}.  
\item [Step 1] [Define $k,a$]. Since $A(q)\ne0$, let $A(q)=aq^{k}+O(q^{k+1})$ with
$a\ne0$.
\item [Step 2] [Define $D$]. Define $D(q)$ by
\begin{equation}\label{eq:defD}
  \frac{aq^{k}B}{A}-\frac{aq^{k}C}{B}=D(q)+O(q^{k+2}),
\end{equation}
where $D(q)$ is a polynomial of degree at most $k+1$; then $D(0)=1$.
\item [Step 3] [Define $A^*,B^*,C^*$]. Let
\begin{align}
  A^*(q)&=\bigl(-D^2A/a+BD\,q^{k}-Ca\,q^{2k}\bigr)/q^{2k+2},\label{eq:AA}\\
  B^*(q)&=2AD/(a\,q^{k})-B,\label{eq:BB}\\
  C^*(q)&=-A\,q^{2}/a.\label{eq:CC}
\end{align}
These are polynomials in $\Q[q]$ satisfying  \eqref{condABC} by \cite[Lemma~3.2]{Han16}.
\end{description}
\end{algo} 

\begin{rem}\label{rem:algo}
In Step~2, if $C=q^2C'$ then
$aq^{k}B/A=D(q)+O(q^{k+2})$; if moreover $A=aq^{k}$, then
$B=D(q)+O(q^{k+2})$.
\end{rem} 

Let us explain how and why algorithm \nabc computes the $H$-fraction of~$F$.

\begin{itemize}[wide]
\item We initialise the data by setting
$F_1:=F$ and $(A_1,B_1,C_1):=(A,B,C)$. From \eqref{Fquad},
$F_1=-A_1/(B_1+C_1F_1)$, so by \eqref{condABC} and Step~1,
\[
  F_1(q)=-a\,q^{k}+O(q^{k+1}),
\]
where $aq^k$ is the lowest term of $A_1$. Hence $-aq^{k}/F_1=1+O(q)$, and there
are a unique polynomial $D$ of degree at most $k+1$ and a unique power series
$F_2$ such that
\begin{equation}\label{eq:onestep}
  \frac{-a\,q^{k}}{F_1}=D(q)-q^{k+2}F_2(q),\qquad\text{i.e.}\qquad
  F_1=\frac{-a\,q^{k}}{D(q)-q^{k+2}F_2(q)} .
\end{equation}
This $D$ is the polynomial of Step~2: indeed
\begin{equation*}
\frac{-a\,q^{k}}{F_1}=\frac{a\,q^{k}B_1}{A_1}+\frac{a\,q^{k}C_1F_1}{A_1}
\end{equation*}
with
$$\frac{a\,q^{k}C_1F_1}{A_1}=\frac{-a\,q^{k}C_1}{B_1+C_1F_1}=\frac{-a\,q^{k}C_1}{B_1}+O(q^{2k+2}),$$
since
$F_1=O(q^{k})$ and $C_1=O(q)$. 
\item Now, apply Step~3 of \nabc: define $A^*,B^*,C^*$ by \eqref{eq:AA}, \eqref{eq:BB}, \eqref{eq:CC}, respectively. According to \cite[Lemma~3.2]{Han16}, for these polynomial conditions \eqref{condABC} still hold and the power
series $F_2$ is still quadratic, namely:
$$A^*+B^*F_2+C^*F_2^2=0.$$
Thus the algorithm can be applied to $F_2$, and so on.
\item Writing $(A_{m+1},B_{m+1},C_{m+1})$ for the triple of $F_{m+1}$, we thus define a
sequence
\begin{equation*}%
  \bigl(A_{m+1},B_{m+1},C_{m+1};\,k_m,a_m,D_m\bigr)
  =\nabc\bigl(A_m,B_m,C_m\bigr),\qquad m≥1,
\end{equation*}
which gives the  iterated version of \eqref{eq:onestep}: $F_m=-a_mq^{k_m}/\bigl(D_m-q^{k_m+2}F_{m+1}\bigr)$. Eventually, we obtain the $H$-fraction of $F$:
\begin{equation}
\begin{split}
  F(q)&=\frac{-a_1q^{k_1}}{D_1}\cfp\frac{a_2\,q^{k_1+k_2+2}}{D_2}\cfp
  \frac{a_3\,q^{k_2+k_3+2}}{D_3}\cfp\cfd\\
  &=\cH\bigl(q;(\sigma_j),(v_j),(u_j)\bigr)
\end{split}\label{HfracABC}
\end{equation}
where we use notation \eqref{eq:Hfracstd} with $v_j=-a_{j+1}$ and $\sigma_j=k_{j+1}=\val A_{j+1}$ for $j≥0$, and $1+u_j(q)q=D_j$ for $j≥1$. 
\end{itemize}

To shorten the notation we will record $k_m$ and $a_m$ simultaneously in the leading
monomial 
$$K_m:=a_mq^{k_m}$$
 of $A_m$, so that Step~1 reads
$A_m=K_m+O(q^{k_m+1})$. In other words, the $H$-fraction \eqref{HfracABC} reads also
\begin{equation}\label{HfracABC2}
  F(q)=\fr{-K_1(q)}{D_1(q)}\cfp\fr{K_2(q)\,q^{k_1+2}}{D_2(q)}\cfp
\fr{K_3(q)\,q^{k_2+2}}{D_3(q)}\cfp\cfd
\end{equation}

\begin{defi}\label{defi:quint}
The sequence of polynomial quintuples
\[
  x_m:=\bigl(A_m,B_m,C_m;K_m,D_m\bigr),\qquad m≥1,
\]
produced by Algorithm \nabc from the quadratic equation \eqref{Fquad} and conditions \eqref{condABC},
is called the \emph{sequence of quintuples associated with~$F$}. By
\eqref{HfracABC} it determines the $H$-fraction of~$F$, and therefore, by
Theorem~\ref{th:super2}(2), all its Hankel determinants.
Note that all the entries of $x_m$ carry the same index $m$, whereas
Algorithm \nabc applied to $(A_m,B_m,C_m)$ returns
\[
  (A_{m+1},B_{m+1},C_{m+1};k_m,a_m,D_m).
\]
\end{defi}

Note that $x_{m+1}$ depends only on $(A_m,B_m,C_m)$. Consequently, if
$(A_{p'},B_{p'},C_{p'})=(A_p,B_p,C_p)$ for some indices $p<p'$, then the
sequence of quintuples is periodic from index $p$ on, with period $p'-p$, and
the $H$-fraction of $F$ is eventually periodic. This is precisely what is going to happen 
for $F=\Phi_n^{(n+2)}$, with $p=2$ and $p'=p_n+2$, where $p_n$ is given in \cref{mainthmA}; see \cref{sec:hfrac}.

\subsection{The $q$-metallic number and its shifted series}

Let us recall some facts about the $q$-metallic power series $\Phi_n$ (see \cite{OP25,HP,Pedon}). Almost by definition \eqref{defPhin} and \eqref{Phin-SE}, the $q$-metallic power series $\Phi_n$ is
characterised by the quadratic functional equation
\begin{equation}\label{eq:FE-Phin}
  q\,\Phi_n(q)^2+\bigl((1+q^n)(1-q)-q[n]_q\bigr)\Phi_n(q)=1 .
\end{equation}
from which follows at once an explicit formula:
\begin{equation*}
	\Phi_n(q)=\frac{1}{2q}\left(q[n]_q+(q^n+1)(q-1)+\sqrt{\bigl(q[n]_q+(q^n+1)(q-1)\bigr)^2+4q}\right).
\end{equation*}
Although we will not use it, we mention that the  Taylor series expansion about zero of $\Phi_n$ is of the form:
\begin{equation*}\label{SE-Phin}
\Phi_n(q)=1+q+\cdots+q^{n-1}+q^{2n}+\sum_{i=2n+1}^{+\infty}\kappa_i(\phi_n) q^i,
\end{equation*}
where coefficients $\kappa_i(\phi_n)$ are integers; see Corollary~2.6 in \cite{OP25}. 

Recall that we are interested in the $\ell$-shifted Hankel determinants of
$\Phi_n$ (for $\ell≥0$), or equivalently in the ordinary Hankel determinants of the shifts $\Phi_n^{(\ell)}$ (definition \eqref{defshift}):
$$\Delta_j^{(\ell)}:=\Delta^{(\ell)}_j(\Phi_n)=\Delta_j\bigl(\Phi_n^{(\ell)}\bigr).$$
Specifically, our target is the shift $\ell=n+2$, i.e.\ the series
$\Phi_n^{(n+2)}$.

The starting point of our whole method is that each shift is \emph{again} a
quadratic power series, obtained from the previous one by an elementary
substitution. More precisely, one has the following statement, easy to check and already proved in \cite[Lemma~7.1]{HP}.

\begin{lem}\label{lem:cut}
Let $F(q)$ be a power series with constant term $f_0$, satisfying a quadratic equation
$$A+BF+CF^2=0$$ 
with $A,B,C$ polynomials,
$C\ne0$. Then the shift $F^{(1)}(q)=\dfrac{F(q)-f_0}{q}$ satisfies the quadratic equation
$$A'+B'F^{(1)}+C'\bigl(F^{(1)}\bigr)^2=0,$$
where
\begin{equation}\label{eq:cut}
  A'=\frac1q\bigl(A+f_0B+f_0^2C\bigr),\qquad
  B'=B+2f_0C,\qquad
  C'=qC .
\end{equation}
\end{lem}

From now on we put:
$$
  Q:=q^2-q+1.
$$
Iterating Lemma~\ref{lem:cut} from \eqref{eq:FE-Phin} produces the quadratic
equation of every shift $\Phi_n^{(\ell)}$. Indeed, rewriting 
$$A^{(0)}+B^{(0)}\Phi_n+C^{(0)}\Phi_n^2=0$$
for \eqref{eq:FE-Phin} in the
normalisation 
$$A^{(0)}=-1,\quad B^{(0)}=(1+q^n)(1-q)-q[n]_q,\quad C^{(0)}=q,$$ 
one
proves by induction (see \cite[Prop.~7.2]{HP}) that
\begin{equation*}
A^{(\ell)}+B^{(\ell)}\Phi_n^{(\ell)}+C^{(\ell)}\bigl(\Phi_n^{(\ell)}\bigr)^2=0,
\end{equation*} 
where, for $0≤\ell≤ n$,
\begin{equation}\label{eq:ABCshift-gen}
  A^{(\ell)}=\frac{q^{\ell+1}-Q\,(q^{n}-q^{n-\ell}+1)}{(q-1)^2},\quad
  B^{(\ell)}=\frac{2q^{\ell+1}-Q\,(q^{n}+1)}{q-1},\quad
  C^{(\ell)}=q^{\ell+1},
\end{equation}
while at $\ell=n+1$
\begin{equation}\label{eq:ABCn+1}
  A^{(n+1)}=-q^{n-1},\qquad
  B^{(n+1)}=-\frac{Q+(q^2-3q+1)q^n}{q-1},\qquad
  C^{(n+1)}=q^{n+2}.
\end{equation}
The case we need, $\ell=n+2$, is one further step of Lemma~\ref{lem:cut}.

\begin{prop}\label{prop:init}
For every integer $n≥1$, the shifted series $F=\Phi_n^{(n+2)}$ is a
quadratic power series with integer coefficients, satisfying
\begin{equation}\label{eq:quad-n+2}
  A^{(n+2)}+B^{(n+2)}\,F+C^{(n+2)}\,F^2=0,
\end{equation}
where
\begin{equation}\label{eq:ABC0}
\begin{dcases*}
  A^{(n+2)}=Q,\quad
  B^{(n+2)}=2q^{3}-q^{2}+q+1,\quad
  C^{(n+2)}=q^{4},&if $n=1$,\\
  A^{(n+2)}=-q^{n-2},\quad
  B^{(n+2)}=-\frac{(q^2-3q+1)\,q^n+Q}{q-1},\quad
  C^{(n+2)}=q^{n+3}&if $n≥2$.
\end{dcases*}
\end{equation}
Moreover $B^{(n+2)},C^{(n+2)}\in\Z[q]$ and the triple
$(A^{(n+2)},B^{(n+2)},C^{(n+2)})$ satisfies the conditions \eqref{condABC}.
\end{prop}

\begin{proof}
We start with case $n≥2$ and apply Lemma~\ref{lem:cut} to  coefficients \eqref{eq:ABCn+1} of the series $F=\Phi_n^{(n+1)}$. The constant term
$f_0=\Phi_n^{(n+1)}(0)$ is determined by evaluating
$A^{(n+1)}+B^{(n+1)}f_0+C^{(n+1)}f_0^2=0$ at $q=0$: since $n≥2$ we have
$A^{(n+1)}(0)=0$, $C^{(n+1)}(0)=0$ and $B^{(n+1)}(0)=1$, whence $f_0=0$.
Thus \eqref{eq:cut} degenerates to
$$
  A^{(n+2)}=\frac1q\,A^{(n+1)}=-q^{n-2},\quad
  B^{(n+2)}=B^{(n+1)},\quad
  C^{(n+2)}=q\,C^{(n+1)}=q^{n+3},
$$
which is the second line of \eqref{eq:ABC0}. The numerator of $B^{(n+2)}$ vanishes at $q=1$, so $B^{(n+2)}\in\Z[q]$.
\\
For $n=1$ the same computation runs one step longer. Formula
\eqref{eq:ABCn+1} still holds, giving 
$$A^{(2)}=-1,\quad B^{(2)}=1+q-q^{2},\quad C^{(2)}=q^{3},$$
 but now the constant term of $\Phi_1^{(2)}$ is $f_0=1$, and
Lemma~\ref{lem:cut} yields  the first line of \eqref{eq:ABC0}. 
\\
In both cases one checks easily 
$A^{(n+2)}\ne0$, $C^{(n+2)}\ne0$, $B^{(n+2)}(0)=1$ and $C^{(n+2)}(0)=0$.
\end{proof}

\section{The Hankel continued fraction of $\Phi_n^{(n+2)}$: proof of Theorem~\ref{mainthmA}}
\label{sec:hfrac}

This section is devoted to the proof of \cref{mainthmA}. Following  the method explained  in \cref{sec:HF-NABC}, we will show that the sequence of quintuples :
\begin{equation}\label{quint}
  x_m=(A_m,B_m,C_m;K_m,D_m),
  \qquad m=1,2,3,\dots,
\end{equation}
associated with $F:=\Phi_n^{(n+2)}$ (in the sense of \cref{defi:quint}) and starting from the triple
\[
  (A_1,B_1,C_1):=\bigl(A^{(n+2)},B^{(n+2)},C^{(n+2)}\bigr)
\]
of Proposition~\ref{prop:init},
is such that 
\begin{equation}\label{quintper}
x_{j+k p_n}=x_{j}\quad\text{for all $j≥2$ and $k≥0$},
\end{equation}
 where $p_n$ is the period defined by
\begin{equation}\label{defpn}
p_n:=
\begin{dcases*}
2&if $n=1$,\\
9&if $n=2$,\\
9n-13&if $n≥3$ is odd,\\
9n-10&if $n≥4$ is even.
\end{dcases*}
\end{equation}
This will imply all assertions of \cref{mainthmA}: the first quintuple $x_1$ determines the pre-periodic block $\cI_n(q)$ in the $H$-fraction of $\Phi_n^{(n+2)}$, while the set of quintuples $x_m$ for $2≤m≤p_n+1$ determines the periodic block $\cP_n(q)$.

This section is organised as follows: \S\ref{ssec:onestep} gives a criterion, valid for any quadratic series, for checking one
step of Algorithm~\nabc; \S\ref{ssec:small} treats $n≤3$;
\S\ref{ssec:tables} describes the tables of the appendices and defines from
them a sequence $(y_m)$ for $n≥4$; \S\ref{ssec:verif} proves that
$(y_m)=(x_m)$; and \S\ref{ssec:proofA} concludes.

\subsection{Verification of one step of Algorithm~\nabc}\label{ssec:onestep}

For a triple $(A,B,C)$ and a quintuple $x=(A,B,C;K,D)$ with $K=a\,q^{k}$,
$a\in\Q^{*}$, $k\in\Z$, $D\in\Q[q]$, we shall use the conditions
\begin{align}
  &A,B,C\in\Q[q],\quad B(0)=1,\quad C(0)=0,\label{eq:adm}\\
  &k≥0,\quad D(0)=1,\quad\deg D≤ k+1.\label{eq:shape}
\end{align}
Condition \eqref{eq:adm} is the input condition \eqref{condABC} of
Algorithm~\nabc apart from $A,C\ne0$, and every output triple $(A^*,B^*,C^*)$ of
Algorithm~\nabc satisfies \eqref{eq:adm}. 

For two quintuples
$x=(A,B,C;K,D)$, $K=aq^k$, and $x'=(A',B',C';K',D')$ we write
\[
  x\to x'\qquad\text{if}\qquad \nabc(A,B,C)=(A',B',C';\,k,a,D),
\]
that is, if $(k,a,D)$ is the output of Steps~1 and~2 of Algorithm~\nabc applied
to $(A,B,C)$, and $(A',B',C')$ is the triple then computed by Step~3. By
construction, the sequence of quintuples associated with $F$ is the unique
sequence $(x_m)$ such that $x_1$ has the triple $(A_1,B_1,C_1)$ and
$x_m\to x_{m+1}$ for all $m≥1$.

Step~3 \eqref{eq:AA}--\eqref{eq:CC} makes sense for any quintuple, whereas
Steps~1 and~2 are truncations, which depend on the sizes of the exponents. The
following criterion replaces them by a polynomiality condition.

\begin{lem}\label{lem:polycrit}
Let $x=(A,B,C;K,D)$, $K=aq^k$, satisfy \eqref{eq:adm} and \eqref{eq:shape},
with $C\ne0$, and define $A^*,B^*,C^*$ by \eqref{eq:AA}--\eqref{eq:CC}. Then
$$\nabc(A,B,C)=(A^*,B^*,C^*;\,k,a,D)$$ if and only if $A^*,B^*\in\Q[q]$ and
$B^*(0)=1$.
\end{lem}

\begin{proof}
The direct implication was noted above. Conversely, assume that $A^*,B^*\in\Q[q]$
and $B^*(0)=1$.

\emph{Step~1.} By \eqref{eq:BB}, $2AD=aq^{k}(B+B^*)$, where $B+B^*$ is a
polynomial with constant term $1+1=2$. Since $D(0)=1$, it follows that
$A=aq^{k}+O(q^{k+1})$. So Step~1 of Algorithm~\nabc, applied to $A$, gives
exactly $k$ and $a$. (In particular $A\ne0$; indeed, if $A=0$ then $B^*=-B$ and
$B^*(0)=-1$.)

\emph{Step~2.} Multiplying \eqref{eq:AA} by $-a/A$ gives
\[
  D^2-\frac{aq^{k}B}{A}\,D+\frac{a^2q^{2k}C}{A}
  =-\frac{aq^{k}}{A}\,q^{k+2}A^* .
\]
Since $aq^{k}/A=1+O(q)$ is a power series and $A^*$ is a polynomial, the
right-hand side is $O(q^{k+2})$; so $D$ is a solution of the congruence
\begin{equation}\label{eq:congD}
  X^2-\frac{aq^{k}B}{A}\,X+\frac{a^2q^{2k}C}{A}\equiv0\pmod{q^{k+2}} .
\end{equation}
Let $\widehat D$ be the polynomial defined by Step~2, i.e.\
$\widehat D\equiv aq^{k}B/A-aq^{k}C/B\pmod{q^{k+2}}$. Since
$aq^{k}B/A\equiv\widehat D+aq^{k}C/B$, the left-hand side of \eqref{eq:congD}
at $X=\widehat D$ is congruent modulo $q^{k+2}$ to
\[
  \widehat D^2-\Bigl(\widehat D+\frac{aq^{k}C}{B}\Bigr)\widehat D
  +\frac{a^2q^{2k}C}{A}
  =aq^{k}C\Bigl(\frac{aq^{k}}{A}-\frac{\widehat D}{B}\Bigr).
\]
Here $aq^kC=O(q^{k+1})$ because $C(0)=0$, and the bracket is $O(q)$ because
$aq^k/A$, $\widehat D$ and $B$ all have constant term $1$. Hence $\widehat D$ is
also a solution of \eqref{eq:congD}. Subtracting the two congruences gives
\[
  (D-\widehat D)\Bigl(D+\widehat D-\frac{aq^{k}B}{A}\Bigr)\equiv0\pmod{q^{k+2}},
\]
and the second factor has constant term $1+1-1=1$, so it is invertible in
$\Q[[q]]$. Therefore $D\equiv\widehat D\pmod{q^{k+2}}$, and $D=\widehat D$
because both have degree at most $k+1$. Thus $(k,a,D)$ is the output of Steps~1
and~2, and Step~3 then produces $(A^*,B^*,C^*)$.
\end{proof}

Since $C^*=-Aq^2/a$ is a polynomial with $C^*(0)=0$ as soon as $A$ is a
polynomial, the lemma gives the following practical test.

\begin{cor}\label{cor:step}
Let $x=(A,B,C;K,D)$ and $x'=(A',B',C';K',D')$ be two quintuples, with $C\ne0$.
Assume that
\begin{enumerate}[leftmargin=1.5cm]
\item[\textup{(S3)}] the triple computed from $x$ by Step~3 is $(A',B',C')$;
\item[\textup{(S12)}] $(A,B,C)$ and $(A',B',C')$ satisfy \eqref{eq:adm}, and
$x$ satisfies \eqref{eq:shape}.
\end{enumerate}
Then $x\to x'$.
\end{cor}

Here, condition (S3) (meaning `Step~3') is an identity between rational functions, whereas (S12) (meaning `Steps~1 \& 2') is a
set of statements about $q$-orders. 

Conversely, if $x\to x'$ then $(A',B',C')$
satisfies \eqref{eq:adm}. Hence, if $x$ satisfies \eqref{eq:shape} and the
triple computed from $x$ by Step~3 does \emph{not} satisfy \eqref{eq:adm}, then
the $K,D$ of $x$ are not those of Algorithm~\nabc.

\subsection{The cases $n=1,2,3$}\label{ssec:small}

For a fixed numerical $n$, Algorithm~\nabc can simply be run from the initial
triple $(A_1,B_1,C_1)=\bigl(A^{(n+2)},B^{(n+2)},C^{(n+2)}\bigr)$ of Proposition~\ref{prop:init}. For $n=1,2,3$ this is
done by the program \texttt{verify\_base.sage}, and the resulting quintuples
are listed in Appendices~\ref{sec:n1}, \ref{sec:n2} and~\ref{sec:n3}. 
In each
case one observes that 
$$(A_{p_n+2},B_{p_n+2},C_{p_n+2})=(A_2,B_2,C_2)$$ with
$p_n=2,9,14$ respectively, so that the sequence is periodic from $m=2$ on, with
period $p_n$; moreover $x_{1+p_n}=(-A_1,B_1,-C_1;\,\ast,\,\ast)$.

\subsection{The pass-tables for $n≥4$}\label{ssec:tables}

From now on $n≥4$. In this subsection we only \emph{describe} the tables of
Appendices~\ref{sec:pass1}--\ref{sec:bridge3} and use them to define a
sequence of quintuples $(y_m)_{m≥1}$; that $y_m=x_m$ is proved in
\S\ref{ssec:verif}. Let us first introduce some notation and vocabulary.

\begin{itemize}
\item %
The entries of the tables are rational functions of
$q$, $q^{2j}$ and $q^{n}$. An index $m$ is written $m=9j+r$ with
$r\in\{1,\dots,9\}$: the \emph{block} $j≥0$ consists of the indices
$9j+1,\dots,9j+9$, and $r$ is the position in the block. We put
\[
  d:=n-2j .
\]
\item A \emph{pass} is a family of nine formulas, the \emph{rows} $r=1,\dots,9$,
giving the quintuple at index $9j+r$ as a function of $(n,j)$. There are two
passes:
\begin{itemize}
\item \emph{pass~$1$} (Appendix~\ref{sec:pass1}), whose rows are given
separately for $j=0$ and $r≤6$ (the \emph{special rows}) and by a
\emph{generic row} otherwise;
\item \emph{pass~$2$} (Appendix~\ref{sec:pass2}), whose rows are all generic;
its index $9j+r$ is local to the pass, and $j$ is an integer when $n$ is even, a
half-integer when $n$ is odd (in both cases $d=n-2j$ is even).
\end{itemize}
\item There are also three \emph{bridges}, which are short lists of quintuples
free of $j$: we call them $\mathrm{b}1,\dots,\mathrm{b}8$ (Appendix~\ref{sec:bridge}),
$\mathrm{c}1,\dots,\mathrm{c}5$ (Appendix~\ref{sec:bridge2}) and
$\mathrm{d}1,\mathrm{d}2,\mathrm{d}3$ (Appendix~\ref{sec:bridge3}). 
All the formulas hold for every $n≥4$, even or odd. 
\item A \emph{tuple} is the value of a
row, or a bridge entry, at a given $(n,j)$. 
\item Inside a pass, the
\emph{successor} of row $r<9$ of block $j$ is row $r+1$ of block $j$, and the
successor of row $9$ of block $j$ is row $1$ of block $j+1$.
\end{itemize}

Now, let us describe the sequences $(y_m)$ in the pass-tables of the Appendices.

\begin{itemize}
\item\emph{The sequence $(y_m)$ for even $n≥4$.} Put $t_n:=\tfrac92n-7$ and recall that
$p_n=9n-10$.
\begin{enumerate}
\item[(a)] For $1≤ m≤ t_n-1$, $y_m$ is row $r$ of pass~$1$ at block $j$,
where $m=9j+r$. Thus pass~$1$ runs through the full blocks $j=0,\dots,\frac
n2-2$ (that is, $d=n,n-2,\dots,4$), and through row~$1$ of the block
$j=\frac n2-1$ ($d=2$).
\item[(b)] $y_{t_n},\dots,y_{t_n+4}$ are $\mathrm{c}1,\dots,\mathrm{c}5$, respectively.
\item[(c)] The next $\tfrac92n-14$ terms are the rows of pass~$2$ from local
index $6$ (row~$6$ at $j=0$) to local index $\tfrac92n-9$ (row~$9$ at
$j=\frac n2-2$, i.e.\ $d=4$), in increasing order.
\item[(d)] The next eight terms are $\mathrm{b}1,\dots,\mathrm{b}8$; thus
$y_{1+p_n}=\mathrm{b}8$.
\item[(e)] $y_{m+p_n}=y_m$ for all $m≥2$.
\end{enumerate}

\item \emph{The sequence $(y_m)$ for odd $n≥5$.} Put $t_n:=\frac{9n-15}2$ and recall that
$p_n=9n-13$.
\begin{enumerate}
\item[(a)] For $1≤ m≤ t_n-1$, $y_m$ is row $r$ of pass~$1$ at block $j$,
where $m=9j+r$. Thus pass~$1$ runs through the full blocks
$j=0,\dots,\frac{n-5}2$ ($d=n,n-2,\dots,5$), and through rows~$1$--$5$ of the
block $j=\frac{n-3}2$ ($d=3$).
\item[(b)] $y_{t_n},y_{t_n+1},y_{t_n+2}$ are $\mathrm{d}1,\mathrm{d}2,\mathrm{d}3$, respectively.
\item[(c)] The next $\frac{9n-29}2$ terms are the rows of pass~$2$, read at
half-integer $j$, from local index $\tfrac{13}2$ (row~$2$ at $j=\tfrac12$) to
local index $\tfrac92n-9$ (row~$9$ at $j=\frac n2-2$, i.e.\ $d=4$), in
increasing order.
\item[(d)] The next eight terms are $\mathrm{b}1,\dots,\mathrm{b}8$; thus
$y_{1+p_n}=\mathrm{b}8$.
\item[(e)] $y_{m+p_n}=y_m$ for all $m≥2$.
\end{enumerate}
\end{itemize}
\medskip

In both cases the counts add up: one period ($m=2,\dots,1+p_n$) consists of
$t_n-2$ tuples of pass~$1$, the tuples of the bridge $\mathrm{c}$ (resp.\
$\mathrm{d}$), the tuples of pass~$2$ and the eight tuples of the bridge
$\mathrm{b}$, i.e.\ 
$$\bigl(\tfrac92n-9\bigr)+5+\bigl(\tfrac92n-14\bigr)+8=9n-10=p_n$$
 for even $n$ and
$$\tfrac{9n-19}2+3+\tfrac{9n-29}2+8=9n-13=p_n$$
 for odd $n$. For a later use, we observe also that
\begin{equation}\label{eq:b8}
  y_1=\bigl(-q^{n-2},\,B_1,\,q^{n+3};\,-q^{n-2},\,D\bigr),\qquad
  \mathrm{b}8=\bigl(q^{n-2},\,B_1,\,-q^{n+3};\,q^{n-2},\,D\bigr),
\end{equation}
with $D=-\frac{Q-q^n}{q-1}$. %
Hence $y_1$ has the triple
$(A_1,B_1,C_1)$ of Proposition~\ref{prop:init}, and $\mathrm{b}8$ has the
triple $(-A_1,B_1,-C_1)$.

\subsection{Verification of the tables}\label{ssec:verif}

\begin{prop}\label{prop:tables}
Let $n≥4$. For every $m≥1$, $y_m\to y_{m+1}$. Hence $y_m=x_m$ for all
$m≥1$.
\end{prop}

\begin{proof}
By condition (e) above (for both parities), it suffices to treat $1≤ m≤ 1+p_n$. We apply
Corollary~\ref{cor:step} to each pair $(y_m,y_{m+1})$, and check its two
conditions.

\emph{Condition \textup{(S3)}.} The pairs $(y_m,y_{m+1})$ are of finitely many
kinds: a row of a pass and its successor (generic rows, special rows, or the
step from the last special row to the generic rows), and the steps involving
a bridge, namely
\begin{align*}
  &\text{pass 1, row 1 at } d=2\to\mathrm{c}1,\qquad
  \mathrm{c}5\to\text{pass 2, row 6 at } j=0
  \qquad(n\text{ even}),\\
  &\text{pass 1, row 5 at } d=3\to\mathrm{d}1,\qquad
  \mathrm{d}3\to\text{pass 2, row 2 at } j=\tfrac12
  \qquad(n\text{ odd}),\\
  &\text{pass 2, row 9 at } d=4\to\mathrm{b}1,\qquad
  \mathrm{b}i\to\mathrm{b}(i{+}1),\qquad
  \mathrm{b}8\to y_2 .
\end{align*}
\begin{itemize}
\item For a step from a row to its successor, (S3) is an identity between rational
functions of $q$, $u:=q^{2j}$ and $w:=q^n$; since it does not depend on the
relative size of the exponents, it is an identity in $\Q(q,u,w)$, valid for
all $n$ and $j$ at once (note $q^{4j}=u^2$ and $q^{n-2j}=w/u$).
\item For a step into
a bridge, (S3) says that the triple computed by Step~3 from the last row of
the pass is that of the bridge entry. By the identity of the first case,
Step~3 maps this row to the triple of the next row, so (S3) amounts to the
identity saying that the bridge entry has the triple of the next row of the
pass at the last block, e.g.\ that $\mathrm{c}1$ has the
triple of row~$2$ of pass~$1$ at $q^{2j}=q^{n-2}$.
\item   For a step out of a bridge
($\mathrm{c}5$, $\mathrm{d}3$, $\mathrm{b}i$), (S3) is a direct identity, the
target row being taken at a fixed $j$.
\end{itemize}
Thus (S3) is an identity in $\Q(q,u,w)$ in the first case and in
$\Q(q,w)$ in the other two; all these identities are proved by the program
\texttt{certify\_step3.sage} by reading the tables from
\texttt{passtables.sage}.

\emph{Condition \textup{(S12)}.} This is a set of statements about
$q$-orders. Write each entry $X$ of a tuple as a finite sum
\[
  X=\sum_{b,c}P_{b,c}(q)\,u^{b}w^{c},\qquad u=q^{2j},\ w=q^{n},\
  P_{b,c}\in\Q(q).
\]
In the tables, $(b,c)$ ranges over $(0,0),(1,0),(-1,1),(0,1)$ --- the terms
$1,q^{2j},q^{n-2j},q^{n}$ --- and the denominator of $P_{b,c}$ has the form
$\lambda\,q^{s}(q-1)^{e}$ with $\lambda\in\Q^{*}$ and $s,e$ two nonnegative integers. The term $(b,c)$ has
$q$-order
\begin{equation}\label{eq:blockord}
  \ord(b,c)=\val_q P_{b,c}+2jb+nc,
\end{equation}
an affine function of $(j,n)$. An entry $X$ is a polynomial in $q$ as soon as
\textup{(a)} the factor $(q-1)^{e}$ of its denominator divides the numerator
after the substitution $u=q^{2j}$, $w=q^{n}$, and \textup{(b)} every term has
$\ord(b,c)≥0$. Condition (a) is checked identically in $j$ and $n$, by
expanding the numerator at $q=1+\varepsilon$ with $j,n$ symbolic (the
coefficients of $\varepsilon^0,\dots,\varepsilon^{e-1}$ are polynomials in
$j,n$, and they vanish). Condition (b), and likewise $B(0)=1$ ($B$ starts with the term $1$
and every other term has order $≥1$), $C(0)=0$,
$k≥0$, $D(0)=1$ and $\deg D≤ k+1$, are inequalities between affine
functions of $(j,n)$, to be proved on a set of the form
$\mathcal{C}=\{\,j≥ j_0,\ n-2j≥ c\,\}$. 

But, for an affine function
$g=\alpha+\beta j+\gamma n$ and a target $\tau\in\Z$, substituting $j=j_0+s$
and $n-2j=c+s'$ with $s,s'≥0$ gives the finite test
\begin{equation}\label{eq:conetest}
  g(j,n)≥\tau\ \text{ on }\mathcal C
  \iff
  \beta+2\gamma≥0,\quad \gamma≥0,\quad
  \alpha+\beta j_0+\gamma(2j_0+c)≥\tau .
\end{equation}
The program \texttt{certify\_steps12.sage} proves in this way, for every tuple
of the tables, that the tuple satisfies \eqref{eq:shape} and that its triple,
as well as the triple computed from it by Step~3, satisfy \eqref{eq:adm}, in
the following ranges: every special row and every bridge entry, for all
$n≥4$; row~$r$ of pass~$1$ (resp.\ pass~$2$), at every $(n,j)$ with
$d≥ c_r$ (resp.\ $d≥ c'_r$) at which it occurs in the definition of
$(y_m)$, where
\[
  (c_1,\dots,c_9)=(1,3,3,2,0,4,4,3,1),\qquad
  (c'_1,\dots,c'_9)=(3,3,2,0,2,4,4,3,1).
\]
No value of $n$ or $j$ is ever substituted.

\emph{Coverage.} It remains to check that every tuple is used in $(y_m)$ only
inside its range. All the bounds $c_r,c'_r$ are at most $4$, and the full
blocks used in (a) and (c) have $d≥4$. The incomplete last block of pass~$1$
uses row~$1$ at $d=2$ (even $n$; $c_1=1$) or rows $1$--$5$ at $d=3$ (odd $n$;
$c_1,\dots,c_5≤3$). The $j$-free tuples are used only for $n≥4$. Hence
Corollary~\ref{cor:step} applies to every pair $(y_m,y_{m+1})$.

Finally $y_1$ has the triple $(A_1,B_1,C_1)$, by \eqref{eq:b8}. Since
$y_m\to y_{m+1}$ for every $m$, induction on $m$ gives $y_m=x_m$.
\end{proof}

\emph{Worked instance.} Take $m=9j+3$ in pass~1 with $j≥1$. The
table (Appendix~\ref{sec:pass1}) reads
\begin{align*}
A&=2q^{d-3},\qquad K=2q^{d-3},\qquad
B=-\frac{Q}{q-1}\,(1+q^{n})-\frac{2(q-2)}{q(q-1)}\,q^{d},\\
C&=\frac{Qq^{2}(q-2)}{2(q-1)^{2}}\,(1+q^{n})+\frac{q(q-2)^{2}}{2(q-1)^{2}}\,q^{d}
   +\frac{Qq^{4}}{2(q-1)^{2}}\,q^{2j},\qquad
D=\frac{q^{d-1}-Q}{q-1}=[d-1]_q-q .
\end{align*}
The orders of the terms of $A$ are $d-3$; of $B$, $0$ (with
$-Q/(q-1)=1+O(q)$), $n$ and $d-1$; of $C$, $2$, $n+2$, $d+1$ and $2j+4$. With
the cancellation of $(q-1)^{e}$, the triple satisfies \eqref{eq:adm} for $d≥3$. Next,
$k=d-3≥0$, $D(0)=1$ and $\deg D≤ d-2=k+1$ (with equality for $d≥4$; for
$d=3$, $D=1$). Step~3 gives
\[
  A^*=-Q\,[n]_q-q^{d-1}+Q\,q^{2j+1},\qquad B^*=Q\,[n]_q+2q^{d-1},
\]
which are the $A,B$ of row~$4$; they are polynomials, and $B^*(0)=1$ for
$d≥2$. Hence the row $\to$ row~$4$ on $\{j≥1,\ d≥3\}$,
i.e.\ $c_3=3$.

\begin{rem}[Why a pass ends]\label{rem:exits}
The bridges are forced by Corollary~\ref{cor:step}. For instance, for even
$n$, at the last block $d=2$ of pass~$1$, row~$1$ $\to$ the quintuple with the
triple of row~$2$; but the triple of row~$3$ does not satisfy \eqref{eq:adm}
there ($A=2q^{d-3}=2q^{-1}$), so the $K,D$ of row~$2$ at $d=2$ are not those of
Algorithm~\nabc (end of \S\ref{ssec:onestep}), and the quintuple at index $t_n$
must be given by another formula: this is $\mathrm{c}1$, which has the same
triple as row~$2$ at $d=2$ but its own $K,D$. 
\\
Similarly, for odd $n$, row~$7$
of pass~$1$ has $A=q^{d-4}$, not a polynomial at $d=3$; and in pass~$2$, row~$2$
has $A=Q\,q^{d-3}$, not a polynomial at $d=2$. The header $n-2j≥ h_r$ of
row~$r$ in the appendices gives the exact range on which the triple of this
generic row satisfies \eqref{eq:adm}, except for row~$1$ of each pass, whose
triple satisfies \eqref{eq:adm} also at $n-2j=-1$ (a value that never occurs).
\end{rem}

\subsection{Proof of Theorem~\ref{mainthmA}}\label{ssec:proofA}

\begin{prop}\label{prop:periodic}
For every $n≥1$, the sequence of quintuples $(x_m)$ associated with
$\Phi_n^{(n+2)}$ is periodic from $m=2$ on, with period $p_n$ (defined in
Theorem~\ref{mainthmA}), and $x_{1+p_n}$ has the triple
$(-A_1,B_1,-C_1)$.
\end{prop}

\begin{proof}
For $n≤3$ see \S\ref{ssec:small}. For $n≥4$, $x_m=y_m$ by
Proposition~\ref{prop:tables}, and $(y_m)$ is periodic from $m=2$ with period
$p_n$ by construction; $x_{1+p_n}=\mathrm{b}8$ has the triple
$(-A_1,B_1,-C_1)$ by \eqref{eq:b8}.
\end{proof}

Notice that the pre-period is exactly $1$: if $x_1$ recurred, it would recur at index
$1+p_n$, but $x_{1+p_n}\ne x_1$.

\begin{proof}[Proof of Theorem~\ref{mainthmA}]
By \eqref{HfracABC2}, the $H$-fraction of $F=\Phi_n^{(n+2)}$ is
\[
  F=\frac{\alpha_0}{\beta_0}\cfp\frac{\alpha_1}{\beta_1}\cfp
  \frac{\alpha_2}{\beta_2}\cfp\cdots,
\]
with $\alpha_0=-K_1$, $\beta_0=D_1$, and
$\alpha_i=K_{i+1}\,q^{k_i+2}$, $\beta_i=D_{i+1}$ for $i≥1$.
For $i≥1$, the term $\alpha_i/\beta_i$ depends only on $x_i$ and $x_{i+1}$.
By Proposition~\ref{prop:periodic}, $x_{i+p_n}=x_i$ for $i≥2$, and
$k_{1+p_n}=\val A_{1+p_n}=\val A_1=k_1$; hence
$\alpha_{i+p_n}/\beta_{i+p_n}=\alpha_i/\beta_i$ for all $i≥1$. This is the
claimed form, with $\cP_n(q)$ the block of the terms $i=1,\dots,p_n$, and
$\cI_n(q)=-K_1/D_1$. 
\\
Moreover, for $n≥2$, \eqref{eq:b8} gives $K_1=-q^{n-2}$ and
$D_1=(q^n-Q)/(q-1)$; for $n=1$, Appendix~\ref{sec:n1} gives $K_1=1$ and
$D_1=1+2q$.
\end{proof}

\section{The $(n+2)$-shifted Hankel determinants: proof of Theorem~\ref{mainthmB}}
\label{sec:Hankel}

\subsection{Proof of Theorem~\ref{mainthmB}.2: the values}

\Cref{mainthmA}  reduces part 2 of Theorem~\ref{mainthmB} to a finite
computation with the data $(\sigma_j,v_j)$ over one period: these are given by finitely many
affine formulas in $n$ and $j$, and we must show that the resulting Hankel determinants, computed with \eqref{eq:HankelDet}, lie in
$\{-2,-1,0,1,2\}$. 

The starting point is Theorem~\ref{th:super2} (part 2) that we apply to the $H$-fraction \eqref{HfracABC} of
$F=\Phi_n^{(n+2)}$ just built in previous section. With the data $(\sigma_j,v_j)$ of \eqref{HfracABC}, formula
\eqref{eq:HankelDet} says that the only nonzero Hankel determinants of $F$ are
the $\De_{s_j}(F)$ at the indices
\begin{equation}\label{eq:sj}
  s_0=0,\qquad s_j=\sigma_0+\sigma_1+\cdots+\sigma_{j-1}+j\quad(j≥1),
\end{equation}
and there
\begin{equation}\label{eq:detformula}
  \De_{s_j}(F)=(-1)^{\epsilon_j}\,v_0^{\,s_j}\,v_1^{\,s_j-s_1}\cdots
  v_{j-1}^{\,s_j-s_{j-1}},\qquad
  \epsilon_j=\sum_{i=0}^{j-1}\binom{\sigma_i+1}{2}.
\end{equation}
All other Hankel determinants of $F$ vanish; these account for the value $0$ in
Theorem~\ref{mainthmB}. It therefore remains to show that each product
\eqref{eq:detformula} equals $\pm1$ or $\pm2$.

Recall that $v_i=-a_{i+1}$ and $\sigma_i=k_{i+1}$. Thus the pair
$(v_i,\sigma_i)$ is read off  (in the pass-tables in the Appendices) from the leading monomial
$K_{i+1}=a_{i+1}q^{k_{i+1}}$ of the quintuple $x_{i+1}$, and it is clear that the partial numerators $v_i$ of the $H$-fraction are all of the form
$$v_i=\pm2^{e_i},\quad\text{with }e_i\in\Z.$$
Therefore the
product \eqref{eq:detformula} is
\begin{equation}\label{eq:2adic}
  \De_{s_j}(F)=(-1)^{\epsilon_j}
  \Bigl(\prod_{i<j}(\sgn v_i)^{\,s_j-s_i}\Bigr)\,2^{E_j},
  \qquad\text{where}\quad
  E_j:=\sum_{i=0}^{j-1}e_i\,(s_j-s_i).
\end{equation}
Thus $\De_{s_j}(F)\in\{\pm 2,\pm 1\}$ if and only if $E_j\in\{0,1\}$ for every
$j$ (the sign is automatically $\pm$). In other words, Theorem~\ref{mainthmB}.2 is equivalent
to the purely arithmetic statement
\begin{equation}\label{eq:Ebound}
  E_j\in\{0,1\}\qquad\text{for all }j≥0 .
\end{equation}

Recall $s_{i+1}-s_i=\sigma_i+1$, and let
\[
  T_i:=\sum_{l=0}^{i}e_l
\]
be the partial sums of the exponent sequence. Reversing the order of summation
in \eqref{eq:2adic},
\begin{equation}\label{eq:EsumdT}
  E_j=\sum_{i=0}^{j-1}e_i\sum_{l=i}^{j-1}(\sigma_l+1)
     =\sum_{l=0}^{j-1}(\sigma_l+1)\sum_{i=0}^{l}e_i
     =\sum_{l=0}^{j-1}(\sigma_l+1)\,T_l .
\end{equation}

\begin{lem}\label{lem:exponent}
For every $n≥1$ the data $(\sigma_i,e_i)$ of the $H$-fraction of
$\Phi_n^{(n+2)}$ satisfy:
\begin{enumerate}[label=\emph{(\roman*)},leftmargin=7ex]
\item $T_i\in\{-1,0,1\}$ for all $i≥0$;
\item $\sigma_i=0$ whenever $T_i\ne0$;
\item the partial sums $\sum_{l=0}^{j-1}T_l$ lie in $\{0,1\}$
for all $j≥0$.
\end{enumerate}
\end{lem}

\begin{proof}
Since $T$ changes only at the indices with
$e_i\ne0$, and $\sum_{l<j}T_l$ changes only at the indices with $T_l\ne0$, it
suffices to list the strings of consecutive indices at which $e_i\ne0$, and to
check that before each string $T=0$ and $\sum T=0$, and that along it
properties (i), (ii) and (iii) hold and both sums return to $0$. In all cases only the following
three strings of exponents occur; we give for each the values of $T$ and of
$\sum_{l≤ i}T_l$ along the string (starting from $0$), and the letters at
which $T\ne0$, where (ii) requires $\sigma=0$:
\begin{equation}\label{eq:strings}
\renewcommand{\arraystretch}{1.15}
\begin{array}{l|l|l|l}
 (e_i) & (T_i) & \bigl(\sum_{l≤ i}T_l\bigr) & \sigma=0\text{ at letters}\\\hline
 1,-2,1 & 1,-1,0 & 1,0,0 & 1,2\\
 1,-2,2,-1,-1,2,-2,1 & 1,-1,1,0,-1,1,-1,0 & 1,0,1,1,0,1,0,0 & 1,2,3,5,6,7\\
 1,-2,2,-2,2,-2,1 & 1,-1,1,-1,1,-1,0 & 1,0,1,0,1,0,0 & 1,\dots,6
\end{array}
\end{equation}
Between two strings all $e_i$ vanish, so $T$ and $\sum T$ stay at $0$.

For $n=1$ all partial numerators are $\pm1$ (Appendix~\ref{sec:n1}), so
$e_i\equiv0$ and there is nothing to prove. For $n=2$, by
Appendix~\ref{sec:n2}, $e_0=0$ and the exponents over one period
($i=1,\dots,9$) are $1,-2,2,-2,2,-2,1,0,0$, the third string of
\eqref{eq:strings}, with $\sigma_1=\dots=\sigma_6=0$. For $n=3$, by
Appendix~\ref{sec:n3}, $e_0=0$ and over one period ($i=1,\dots,14$) the
nonzero exponents are two copies of the first string, at $i=1,2,3$ and
$i=10,11,12$, with $\sigma_1=\sigma_2=\sigma_{10}=\sigma_{11}=0$.

Let $n≥4$. We follow the sequence $(x_m)=(y_m)$ of \S\ref{ssec:tables}, and
read $a_m$ (hence $e_{m-1}$, defined by $a_m=\pm2^{e_{m-1}}$) and $k_m$
($=\sigma_{m-1}$) from the values of $K_m=a_mq^{k_m}$ in the tables. In every
tuple, $a$ is a constant, independent of $n$ and $j$.

\emph{Generic rows.} For the generic rows of a pass, $a$ depends only on the
residue $r$ (for pass~$2$, $r$ is the residue in the index $9j+r$ of
Appendix~\ref{sec:pass2}), and $k$ is affine in $j$ and $d=n-2j$:
\[
\arraycolsep=3pt
\renewcommand{\arraystretch}{1.1}
\begin{array}{r|ccccccccc}
 r & 1&2&3&4&5&6&7&8&9\\\hline
 a,\ \text{pass 1} & 2 & -\tfrac14 & 2 & -1 & 1 & 1 & 1 & 1 & -1\\
 k,\ \text{pass 1} & 0&0&d{-}3&0&2j{-}1&1&d{-}4&0&2j\\
 e,\ \text{pass 1} & 1&-2&1&0&0&0&0&0&0\\\hline
 a,\ \text{pass 2} & 1 & 1 & 1 & -1 & 2 & -\tfrac14 & 2 & -1 & 1\\
 k,\ \text{pass 2} & 1&d{-}3&0&2j{-}1&0&0&d{-}4&0&2j\\
 e,\ \text{pass 2} & 0&0&0&0&1&-2&1&0&0
\end{array}
\]
So in a full block of either pass the nonzero exponents form the first string
of \eqref{eq:strings} (rows $1,2,3$ of pass~$1$, rows $5,6,7$ of pass~$2$), with
$k=0$ at its first two letters.

\emph{The other tuples.} The rows $m=1,\dots,6$ of pass~$1$ at $j=0$ and the
bridges have
\[
\arraycolsep=3pt
\renewcommand{\arraystretch}{1.1}
\begin{array}{r|cccccc|ccccc|ccc}
 & m{=}1&2&3&4&5&6 & \mathrm{c}1&\mathrm{c}2&\mathrm{c}3&\mathrm{c}4&\mathrm{c}5
 & \mathrm{d}1&\mathrm{d}2&\mathrm{d}3\\\hline
 a & -1&2&\tfrac14&2&-1&-1 & -\tfrac14&4&\tfrac12&\tfrac12&4 & 1&-1&-1\\
 k & n{-}2&0&0&n{-}3&0&1 & 0&0&n{-}3&0&0 & 1&n{-}1&1\\
 e & 0&1&-2&1&0&0 & -2&2&-1&-1&2 & 0&0&0
\end{array}
\]
\[
\arraycolsep=3pt
\renewcommand{\arraystretch}{1.1}
\begin{array}{r|cccccccc}
 & \mathrm{b}1&\mathrm{b}2&\mathrm{b}3&\mathrm{b}4&\mathrm{b}5&\mathrm{b}6&\mathrm{b}7&\mathrm{b}8\\\hline
 a & 1&-1&-1&2&\tfrac14&2&1&1\\
 k & 1&0&n{-}3&0&0&n{-}2&n{+}1&n{-}2\\
 e & 0&0&0&1&-2&1&0&0
\end{array}
\]

\emph{Along the sequence.} At $m=1$, $e_0=0$. The rows $m=2,3,4$ give the first
string, with $k=0$ at $m=2,3$, and the subsequent tuples of pass~$1$ give zeros
and copies of the first string, as above. For odd $n$ this goes on up to row~$5$
of the last block ($d=3$); then $\mathrm{d}1,\mathrm{d}2,\mathrm{d}3$ and the
rows $2,3,4$ of pass~$2$ at $j=\tfrac12$ give zeros. For even $n$, the last block
of pass~$1$ ($d=2$) contributes only row~$1$, and row~$1$ at $d=2$,
$\mathrm{c}1,\dots,\mathrm{c}5$ and the rows~$6$, $7$ of pass~$2$ at $j=0$ give
the exponents $1,-2,2,-1,-1,2,-2,1$, the second string of \eqref{eq:strings},
with $k=0,0,0,n-3,0,0,0,n-4$, hence $k=0$ at letters $1,2,3,5,6,7$. In both
cases pass~$2$ then gives zeros and copies of the first string, up to row~$9$ of
its last full block. Finally $\mathrm{b}1,\dots,\mathrm{b}8$ give
$0,0,0,1,-2,1,0,0$, the first string with $k=0$ at $\mathrm{b}4$,
$\mathrm{b}5$, and the sequence returns to $m=2$. This covers one period, and
the lemma follows.
\end{proof}

\cref{mainthmB}.2 is an immediate consequence of this lemma. Indeed, by (ii), $(\sigma_l+1)T_l=T_l$ for every $l$, so \eqref{eq:EsumdT} collapses to 
\begin{equation}\label{eq:ejtj}
E_j=\sum_{l=0}^{j-1}T_l,
\end{equation}
which lies in $\{0,1\}$ by
(iii). Thus \eqref{eq:Ebound} holds.

\subsection{Proof of Theorem~\ref{mainthmB}.1: the (anti-)periodicity}\label{ssec:antiper}

In this section we explain how the $p_n$-periodicity of the $H$-fraction of $\Phi_n^{(n+2)}$ --- or, equivalently, of the associated sequence of quintuples $x_m=(A_m,B_m,C_m;K_m,D_m)$ --- yields the $2n(n+1)$-periodicity (if $n$ is even) or $2n(n+1)$-antiperiodicity (if $n$ is odd) of the Hankel determinant sequence $\De^{(n+2)}$.

As usual, we let $n≥1$ be any integer and, to make the text easier to read, we put 
$$L:=2n(n+1)\quad\text{and}\quad P:=p_n.$$
By Proposition~\ref{prop:periodic},
$x_{m+P}=x_m$ for $m≥2$, and $x_{1+P}$ has the triple $(-A_1,B_1,-C_1)$.
Changing the signs of $A$ and $C$ does not change $k$, changes $a$ into $-a$,
and does not change $D$ in Step~2 of Algorithm~\nabc. Hence $K_{1+P}=-K_1$, and in
terms of the $H$-fraction this means:
\begin{equation}\label{eq:perdata}
  \sigma_{i+P}=\sigma_i\quad(i≥0),\qquad v_{i+P}=v_i\quad(i≥1),\qquad
  v_P=-v_0 .
\end{equation}
In particular,
\begin{equation}\label{eq:siP}
s_{i+P}=s_i+s_P,\qquad \epsilon_{i+P}=\epsilon_i+\epsilon_P,
\qquad e_{i+P}=e_i\quad\text{for all }i≥0.
\end{equation}

\begin{lem}\label{lem:sP} 
One has: $s_P=L$ and $s_{i+P}=s_i+L$ for all $i≥0$.
\end{lem}

\begin{proof}
Because of \eqref{eq:siP} it suffices to prove first assertion. 
By definition, 
$$s_P=\sum_{m=0}^{P-1}(\sigma_m+1)=\sum_{m=1}^{P}(k_m+1).$$
For $n≥4$,
summing the values of $k$ in the tables segment by segment (one full block of
either pass contributes $9+(2n-7)=2n+2$, by the $k$-table in the proof of
Lemma~\ref{lem:exponent}) gives
\[
\renewcommand{\arraystretch}{1.1}
\begin{array}{l|l|l}
 \text{segment} & n\text{ even} & n\text{ odd}\\\hline
 m=1,\dots,6\ (j=0) & 2n+2 & 2n+2\\
 \text{pass 1, rows }7,8,9\text{ at }j=0 & n-1 & n-1\\
 \text{full blocks of pass 1} & (\tfrac n2-2)(2n+2) & \tfrac{n-5}2(2n+2)\\
 \text{last block of pass 1} & 1\ \text{(row 1)} & n+1\ \text{(rows 1--5)}\\
 \text{bridge c, resp.\ d} & n+2 & n+4\\
 \text{pass 2 before its full blocks} & n\ \text{(rows 6--9, } j=0) & 2n\ \text{(rows 2--9, } j=\tfrac12)\\
 \text{full blocks of pass 2} & (\tfrac n2-2)(2n+2) & \tfrac{n-5}2(2n+2)\\
 \mathrm{b}1,\dots,\mathrm{b}7 & 3n+4 & 3n+4\\\hline
 \text{total} & 2n(n+1) & 2n(n+1)
\end{array}
\]
Hence, for both parities of $n≥4$ we obtain $s_P=2n(n+1)=L$.
For $n=1,2,3$ the sum over the tables of
Appendices~\ref{sec:n1}--\ref{sec:n3} gives easily the same result. 
\end{proof} 

Next, we connect the values of $\De_{s_i+L}(F)$ and $\De_{s_i}(F)$.
\begin{lem}\label{lem:shiftP}
Put $\Pi:=\prod_{l=0}^{P-1}v_l$. We have: $\De_{s_i+L}(F)=(-\Pi)^{s_i}\,\De_L(F)\,\De_{s_i}(F)$
for all $i≥0$.
\end{lem}

\begin{proof}
By definition \eqref{eq:detformula},
$
\De_{s_{i+P}}=(-1)^{\epsilon_{i+P}}\,
  v_0^{\,s_{i+P}}v_1^{\,s_{i+P}-s_1}\cdots v_{i+P-1}^{\,s_{i+P}-s_{i+P-1}}
$.
We split this expression into two parts, by setting
\begin{equation*}
\de:=(-1)^{\epsilon_i+\epsilon_P}\,v_0^{\,s_{i+P}}v_1^{\,s_{i+P}-s_1}\cdots v_{P-1}^{\,s_{i+P}-s_{P-1}},
\qquad\de':=v_{P}^{\,s_{i+P}-s_{P}}\cdots v_{i+P-1}^{\,s_{i+P}-s_{i+P-1}}.
\end{equation*}
Since $s_{i+P}-s_l=s_i+(s_P-s_l)$ and $\epsilon_{i+P}=\epsilon_P+\epsilon_i$ by \eqref{eq:siP}, one has 
$$\de=(-1)^{\epsilon_i+\epsilon_P}\,\Pi^{s_i}\prod_{l=0}^{P-1}v_l^{s_P-s_l}.$$
On the other hand, using \eqref{eq:perdata} and again \eqref{eq:siP}, one gets
\begin{equation*}
\de'=\prod_{l=0}^{i-1}v_{l+P}^{s_{i+P}-s_{l+P}}=(-v_0)^{s_i}\prod_{l=1}^{i-1}v_{l}^{s_{i}-s_{l}}=(-1)^{s_i}\prod_{l=0}^{i-1}v_{l}^{s_{i}-s_{l}}.
\end{equation*}
Gathering all together gives
\begin{align*}
\De_{s_{i+P}}&=(-1)^{\epsilon_{P}}\,\Pi^{s_i}\prod_{l=0}^{P-1}v_l^{s_P-s_l}\times (-1)^{\epsilon_{i}}\,(-1)^{s_i}\prod_{l=0}^{i-1}v_{l}^{s_{i}-s_{l}}\\
&=(-\Pi)^{s_i}\,\De_{s_P}(F)\,\De_{s_i}(F),
\end{align*}
which, by \cref{lem:sP}, is the expected result.
\end{proof}

Our last step consists in computing the constants in \cref{lem:shiftP}.
\begin{lem}\label{lem:sign}
For every $n≥1$: \textup{(i)} $\Pi=-1$; \textup{(ii)} $\De_L(F)=(-1)^n$.
\end{lem}

\begin{proof}
\emph{\textup{(i)}} Since $e_{l+P}=e_l$ by \eqref{eq:siP}, we have $T_{l+P}=T_{P-1}+T_l$ for
$l≥0$, hence 
$$T_{l+mP}=mT_{P-1}+T_l$$
 for any $l,m≥0$. But $T$ is bounded by \cref{lem:exponent}(i), so that $T_{P-1}=0$, i.e.\ $|\Pi|=2^{T_{P-1}}=1$.
On the other hand, as $v_l=-a_{l+1}$, the sign of $\Pi$ is
$$(-1)^{P}\prod_{l=1}^P\sgn a_l=(-1)^{P-\#\{1≤l≤ P:\ a_l<0\}}=(-1)^{\#\{1≤l≤ P:\ a_l>0\}}.$$
For $n≥4$ we count the positive $a_l$, $l=1,\dots,P$, in the tables: a full
block of either pass has $6$ of them, and the remaining segments give
\[
\renewcommand{\arraystretch}{1.1}
\begin{array}{l|cccccc|c}
 & l≤6 & \text{rows }7\text{--}9 & \text{last block} & \text{bridge c/d} & \text{pass 2 start} & \mathrm{b}1\text{--}\mathrm{b}7 & \text{total}\\\hline
 n\text{ even} & 3&2&1&4&2&5&17\\
 n\text{ odd} & 3&2&3&1&5&5&19
\end{array}
\]
which is odd in both cases. Thus $\sgn(\Pi)=-1$. 
\\
For $n=1,2,3$ the sign is read off directly from
Appendices~\ref{sec:n1}--\ref{sec:n3}.

\emph{\textup{(ii)}} By \eqref{eq:ejtj}, $E_j=\sum_{0≤l<j}T_l$. Since
$T_{l+P}=T_l$ by the proof of (i), one gets
$$E_{j+P}=\sum_{l=0}^{P-1}T_l+\sum_{l=P}^{P+j-1}T_l=E_P+E_j$$ for all $j$, hence $E_{j+mP}=mE_P+E_j$ for all $m≥0$. Then, as above, the fact that $E$ is
bounded forces $E_P=0$. By \eqref{eq:2adic}, 
it follows that
\begin{equation*}
  \De_L(F)=\De_{s_P}(F)=(-1)^{\epsilon_P}\prod_{l=0}^{P-1}(\sgn v_l)^{\,s_P-s_l}
\end{equation*}
and eventually, since $s_P=L$ is even:
\begin{equation*}
\De_L(F)=(-1)^{\theta(n)}\quad\text{with}\quad \theta(n):=\epsilon_P+\sum_{\substack{0≤l≤P-1\\ v_l<0}}s_l .
\end{equation*}
For $n≥4$, write $n=2t$ or $n=2t+1$. Along one period every $k_m$ is an
affine function of $t$ and of the block number $j$ (read off the values of $K$),
so every $s_l=\sum_{m=1}^l (k_m+1)$ is a polynomial in $t$ and $j$, and $\theta(n)$ is obtained by exact
summation over the segments above. The companion program
\texttt{certify\_sign.sage} carries this out symbolically and finds
\[
  \theta(2t)=\tfrac{160}3t^3-24t^2+\tfrac23t,\qquad
  \theta(2t+1)=\tfrac{160}3t^3+40t^2+\tfrac23t-1 .
\]
In both cases $(\theta(n)-n)/2$ is a polynomial in $t$ of degree $3$ taking integer
values at $t=0,1,2,3$, hence at every integer (Newton's forward-difference
formula); so $\theta(n)\equiv n\pmod2$ and (ii) is proved when $n≥4$.\\
 For $n=1,2,3$, (ii) is a simple finite computation
(\texttt{verify\_base.sage}).
\end{proof}

We can now conclude the proof of the (anti-)periodicity statement in Theorem~\ref{mainthmB}.  By \cref{lem:shiftP} and \cref{lem:sign},
$$\De_{s_i+L}(F)=(-1)^n\De_{s_i}(F)$$
 for all $i≥0$. Since $s_{i+P}=s_i+L$ by \cref{lem:sP}, an
index $j≥0$ is of the form $s_i$ if and only if $j+L$ is, and the other
indices give $0$ on both sides. Finally, we have proved that 
$$\De^{(n+2)}_{j+2n(n+1)}=(-1)^n\De^{(n+2)}_j$$
for all $j≥0$, the aimed result.

\begin{rem}\label{rem:LIN}
In a forthcoming paper we show that $\De^{(n+2)}$ is a signed difference of
two shifts of the ordinary Hankel sequence: for all $n≥1$ and $j≥0$,
\[
  \De_j^{(n+2)}=\epsilon_n(-1)^{nj}\Bigl(\De^{(0)}_{j+n(n+1)+1}
  +(-1)^{n+1}\De^{(0)}_{j+(n+1)(n+2)}\Bigr),
  \qquad
  \epsilon_n=\begin{cases}1,&4\mid n,\\-1,&4\nmid n.\end{cases}
\]
Combined with the $2n(n+1)$-(anti)periodicity of $\De^{(0)}$
\cite{HP}, this gives an alternative proof of the periodicity part of
Theorem~\ref{mainthmB}.
\end{rem}

\appendix

\section{Base case $n=1$}\label{sec:n1}

As for all cases, the initial triple is given by Proposition~\ref{prop:init}. The sequence has pre-period $1$
and period $2$ (from $m=2$), which means that $m=3$ maps back to $m=2$; the three
transitions are identities between concrete polynomials in $q$, computed in
\texttt{verify\_base.sage}. Note that all
$a_m=\pm1$ and that
$\sum_m\bigl(k_m+1\bigr)=4=2n(n+1)$ over one period.

For all $m≥1$ one has $B_m=2q^{3}-q^{2}+q+1$; the other data are listed below:

\smallskip\noindent $m=1$:
\begin{align*}
A_m &= Q, \qquad C_m = q^{4}, \qquad K_m = 1, \qquad D_m = 2\,q+1
\end{align*}
\smallskip\noindent $m=2$:
\begin{align*}
A_m &= -q^{2}, \qquad C_m = -Q q^{2}, \qquad K_m = -q^{2}, \qquad D_m = B_m
\end{align*}
\smallskip\noindent $m=3$:
\begin{align*}
A_m &= -Q, \qquad C_m = -q^{4}, \qquad K_m = -1, \qquad D_m = 2\,q+1
\end{align*}

\section{Base case $n=2$}\label{sec:n2}

This appendix lists the $10$ tuples
$\bigl(A_m,B_m,C_m;K_m,D_m\bigr)$ for $m=1,\dots,10$,
computed by \texttt{verify\_base.sage}. The period is $9$ (from $m=2$), and
$m=10$ maps back to $m=2$.

\smallskip
\noindent $m=1$:
\begin{align*}
A_m &= -1, \qquad B_m = -q^{3} + 2 \, q^{2} + 1, \qquad C_m = q^{5}, \qquad K_m = -1, \qquad D_m = 1
\end{align*}
\smallskip\noindent $m=2$:
\begin{align*}
A_m &= q^{3} - q + 2, \qquad B_m = q^{3} - 2 \, q^{2} + 1, \qquad C_m = -q^{2},\\
 K_m &= 2, \qquad D_m = \frac{1}{2} \, q + 1
\end{align*}
\smallskip\noindent $m=3$:
\begin{align*}
A_m &= -\frac{1}{8} \, q^{3} - \frac{3}{8} \, q + \frac{1}{4}, \qquad B_m = \frac{1}{2} \, q^{4} + \frac{3}{2} \, q^{2} + 1, \qquad C_m = -\frac{1}{2} \, q^{5} + \frac{1}{2} \, q^{3} - q^{2}\\
K_m &= \frac{1}{4}, \qquad D_m = \frac{3}{2} \, q + 1
\end{align*}
\smallskip\noindent $m=4$:
\begin{align*}
A_m &= 2 \, q^{3} + 2 \, q^{2} + 6 \, q + 4, \qquad B_m = -2 \, q^{4} - q^{3} - 6 \, q^{2} + 1, \qquad C_m = \frac{1}{2} \, q^{5} + \frac{3}{2} \, q^{3} - q^{2}\\
K_m &= 4, \qquad D_m = -\frac{3}{2} \, q + 1
\end{align*}
\smallskip\noindent $m=5$:
\begin{align*}
A_m &= -\frac{1}{8} \, q^{3} - \frac{1}{8} \, q^{2} - \frac{3}{8} \, q - \frac{1}{4}, \qquad B_m = \frac{1}{2} \, q^{4} + \frac{1}{2} \, q^{3} + \frac{5}{2} \, q^{2} + 1\\
C_m &= -\frac{1}{2} \, q^{5} - \frac{1}{2} \, q^{4} - \frac{3}{2} \, q^{3} - q^{2}, \qquad K_m = -\frac{1}{4}, \qquad D_m = -\frac{3}{2} \, q + 1
\end{align*}
\smallskip\noindent $m=6$:
\begin{align*}
A_m &= -2 \, q^{3} - 6 \, q + 4, \qquad B_m = -2 \, q^{4} - q^{3} - 6 \, q^{2} + 1, \qquad C_m = -\frac{1}{2} \, q^{5} - \frac{1}{2} \, q^{4} - \frac{3}{2} \, q^{3} - q^{2}\\
K_m &= 4, \qquad D_m = \frac{3}{2} \, q + 1
\end{align*}
\smallskip\noindent $m=7$:
\begin{align*}
A_m &= \frac{1}{8} \, q^{3} - \frac{1}{8} \, q + \frac{1}{4}, \qquad B_m = \frac{1}{2} \, q^{4} + \frac{3}{2} \, q^{2} + 1, \qquad C_m = \frac{1}{2} \, q^{5} + \frac{3}{2} \, q^{3} - q^{2}, \qquad K_m = \frac{1}{4}\\
D_m &= \frac{1}{2} \, q + 1
\end{align*}
\smallskip\noindent $m=8$:
\begin{align*}
A_m &= 2, \qquad B_m = q^{3} - 2 \, q^{2} + 1, \qquad C_m = -\frac{1}{2} \, q^{5} + \frac{1}{2} \, q^{3} - q^{2}, \qquad K_m = 2, \qquad D_m = 1
\end{align*}
\smallskip\noindent $m=9$:
\begin{align*}
A_m &= q^{3}, \qquad B_m = -q^{3} + 2 \, q^{2} + 1, \qquad C_m = -q^{2}, \qquad K_m = q^{3}\\
D_m &= -q^{3} + 2 \, q^{2} + 1
\end{align*}
\smallskip\noindent $m=10$:
\begin{align*}
A_m &= 1, \qquad B_m = -q^{3} + 2 \, q^{2} + 1, \qquad C_m = -q^{5}, \qquad K_m = 1, \qquad D_m = 1
\end{align*}

\section{Base case $n=3$}\label{sec:n3}

This appendix lists the $15$ tuples for $m=1,\dots,15$, computed by
\texttt{verify\_base.sage}. The period is $14$ (from $m=2$), and $m=15$ maps
back to $m=2$.

\smallskip\noindent $m=1$:
\begin{align*}
A_m &= -q, \qquad B_m = -q^{4} + 2 \, q^{3} + q^{2} + 1, \qquad C_m = q^{6}, \qquad K_m = -q, \qquad D_m = q^{2} + 1
\end{align*}
\smallskip\noindent $m=2$:
\begin{align*}
A_m &= q^{4} - q^{3} + 2 \, q^{2} - q + 2, \qquad B_m = q^{4} - 2 \, q^{3} + q^{2} + 1, \qquad C_m = -q^{3}, \qquad K_m = 2\\
D_m &= \frac{1}{2} \, q + 1
\end{align*}
\smallskip\noindent $m=3$:
\begin{align*}
A_m &= -\frac{1}{8} \, q^{4} + \frac{1}{8} \, q^{3} - \frac{1}{4} \, q^{2} + \frac{1}{8} \, q + \frac{1}{4}, \qquad B_m = \frac{1}{2} \, q^{5} - \frac{1}{2} \, q^{4} + 2 \, q^{3} + \frac{1}{2} \, q^{2} + 1\\
C_m &= -\frac{1}{2} \, q^{6} + \frac{1}{2} \, q^{5} - q^{4} + \frac{1}{2} \, q^{3} - q^{2}, \qquad K_m = \frac{1}{4}, \qquad D_m = -\frac{1}{2} \, q + 1
\end{align*}
\smallskip\noindent $m=4$:
\begin{align*}
A_m &= 2, \qquad B_m = -q^{4} - 3 \, q^{2} + 1, \qquad C_m = \frac{1}{2} \, q^{6} - \frac{1}{2} \, q^{5} + q^{4} - \frac{1}{2} \, q^{3} - q^{2}, \qquad K_m = 2\\
D_m &= 1
\end{align*}
\smallskip\noindent $m=5$:
\begin{align*}
A_m &= -q^{4} + q^{3} - 3 \, q^{2} + q - 1, \qquad B_m = q^{4} + 3 \, q^{2} + 1, \qquad C_m = -q^{2}, \qquad K_m = -1\\
D_m &= q + 1
\end{align*}
\smallskip\noindent $m=6$:
\begin{align*}
A_m &= -q^{4} - q^{2} - q, \qquad B_m = 2 \, q^{5} - q^{4} + 4 \, q^{3} + q^{2} + 1\\
C_m &= -q^{6} + q^{5} - 3 \, q^{4} + q^{3} - q^{2}, \qquad K_m = -q, \qquad D_m = 2 \, q^{2} - q + 1
\end{align*}
\smallskip\noindent $m=7$:
\begin{align*}
A_m &= -q^{4} + q^{3} - q^{2}, \qquad B_m = 2 \, q^{5} - q^{4} + 2 \, q^{3} + q^{2} + 1, \qquad C_m = -q^{6} - q^{4} - q^{3}\\
K_m &= -q^{2}, \qquad D_m = 2 \, q^{3} + q^{2} + q + 1
\end{align*}
\smallskip\noindent $m=8$:
\begin{align*}
A_m &= -q^{4} - q^{2} - q, \qquad B_m = 2 \, q^{5} - q^{4} + 2 \, q^{3} + q^{2} + 1, \qquad C_m = -q^{6} + q^{5} - q^{4}\\
K_m &= -q, \qquad D_m = 2 \, q^{2} - q + 1
\end{align*}
\smallskip\noindent $m=9$:
\begin{align*}
A_m &= -q^{4} + q^{3} - 3 \, q^{2} + q - 1, \qquad B_m = 2 \, q^{5} - q^{4} + 4 \, q^{3} + q^{2} + 1, \qquad C_m = -q^{6} - q^{4} - q^{3}\\
K_m &= -1, \qquad D_m = q + 1
\end{align*}
\smallskip\noindent $m=10$:
\begin{align*}
A_m &= -1, \qquad B_m = q^{4} + 3 \, q^{2} + 1, \qquad C_m = -q^{6} + q^{5} - 3 \, q^{4} + q^{3} - q^{2}, \qquad K_m = -1\\
D_m &= 1
\end{align*}
\smallskip\noindent $m=11$:
\begin{align*}
A_m &= -q^{4} + q^{3} - 2 \, q^{2} + q + 2, \qquad B_m = -q^{4} - 3 \, q^{2} + 1, \qquad C_m = -q^{2}, \qquad K_m = 2\\
D_m &= -\frac{1}{2} \, q + 1
\end{align*}
\smallskip\noindent $m=12$:
\begin{align*}
A_m &= \frac{1}{8} \, q^{4} - \frac{1}{8} \, q^{3} + \frac{1}{4} \, q^{2} - \frac{1}{8} \, q + \frac{1}{4}, \qquad B_m = \frac{1}{2} \, q^{5} - \frac{1}{2} \, q^{4} + 2 \, q^{3} + \frac{1}{2} \, q^{2} + 1\\
C_m &= \frac{1}{2} \, q^{6} - \frac{1}{2} \, q^{5} + q^{4} - \frac{1}{2} \, q^{3} - q^{2}, \qquad K_m = \frac{1}{4}, \qquad D_m = \frac{1}{2} \, q + 1
\end{align*}
\smallskip\noindent $m=13$:
\begin{align*}
A_m &= 2 \, q, \qquad B_m = q^{4} - 2 \, q^{3} + q^{2} + 1, \qquad C_m = -\frac{1}{2} \, q^{6} + \frac{1}{2} \, q^{5} - q^{4} + \frac{1}{2} \, q^{3} - q^{2}\\
K_m &= 2 \, q, \qquad D_m = q^{2} + 1
\end{align*}
\smallskip\noindent $m=14$:
\begin{align*}
A_m &= q^{4}, \qquad B_m = -q^{4} + 2 \, q^{3} + q^{2} + 1, \qquad C_m = -q^{3}, \qquad K_m = q^{4}\\
D_m &= -q^{4} + 2 \, q^{3} + q^{2} + 1
\end{align*}
\smallskip\noindent $m=15$:
\begin{align*}
A_m &= q, \qquad B_m = -q^{4} + 2 \, q^{3} + q^{2} + 1, \qquad C_m = -q^{6}, \qquad K_m = q, \qquad D_m = q^{2} + 1
\end{align*}

\section{Pass $1$}\label{sec:pass1}

This appendix lists the tuples of pass~$1$, indexed by $m=9j+r$
($r\in\{1,\dots,9\}$), starting from the initial tuple $y_1$ of \eqref{eq:b8} at
$m=1$. The formulas hold for every $n≥4$, even or odd. Each generic row is written with the range on which its triple
$(A,B,C)$ satisfies \eqref{eq:adm}, i.e.\ is polynomial with $B(0)=1$, $C(0)=0$
(see Remark~\ref{rem:exits} for the exact ranges).\footnote{This
is stronger than polynomiality: rows $4$, $8$, $9$ of pass~$1$ and rows $3$, $8$,
$9$ of pass~$2$ are still polynomial at $n-2j=h_r-1$, but there $C(0)\ne0$
(rows $4$, $8$ of pass~$1$ have $C=-1$, row~$9$ has $C(0)=1$; rows $3$, $8$ of
pass~$2$ have $C=-Q$, row~$9$ has $C(0)=-1$).} A row $\to$ the next one exactly where both triples satisfy \eqref{eq:adm}
(see \S\ref{ssec:onestep}), which determines where pass~$1$ ends
(Remark~\ref{rem:exits}). The rows $j=0$ are used for every $n≥4$.

\smallskip\noindent $m=1$, $m=9j+1$,  $j=0$:
\begin{align*}
A_m &= -q^{n - 2}, \qquad B_m = -\frac{{\left(q^{2} - 3 \, q + 1\right)} q^{n} + Q}{q - 1}, \qquad C_m = q^{n + 3}, \qquad K_m = -q^{n - 2}\\
D_m &= -\frac{Q - q^{n}}{q - 1}
\end{align*}
\smallskip\noindent $m=9j+1$, $j≥1$, $n-2j≥0$:
\begin{align*}
A_m &= -\frac{Q(q-2)^{2} q^{4j} + Q(q-2) q^{2j}\left(1+q^{n}\right) + q^{n+1}}{(q-1)^{2} q^{2j}}, \qquad B_m = -\frac{Q\left(2(q-2) q^{2j} + q^{n} + 1\right)}{q - 1}\\
C_m &= -Q q^{2 \, j}, \qquad K_m = 2, \qquad D_m = -\frac{1}{2} \, q + 1
\end{align*}
\smallskip\noindent $m=2$, $m=9j+2$,  $j=0$:
\begin{align*}
A_m &= \frac{\left(q^{2} - 2 q + 2\right) q^{n} + Q(q-2)}{q - 1}, \qquad B_m = \frac{{\left(q^{2} - 3 \, q + 3\right)} q^{n} - Q}{q - 1}, \qquad C_m = -q^{n}\\
K_m &= 2, \qquad D_m = \frac{1}{2} \, q + 1
\end{align*}
\smallskip\noindent $m=9j+2$, $j≥1$, $n-2j≥1$:
\begin{align*}
A_m &= \frac{Q q(q-2) q^{2j}\left(1+q^{n}\right) + Q q^{3} q^{4j} + (q-2)^{2} q^{n}}{8 q (q-1)^{2} q^{2j}}\\
B_m &= \frac{Q q^{2}(q-2) q^{4j} + Q\left(q^{2}-2q+2\right) q^{2j}\left(1+q^{n}\right) + q(q-2) q^{n}}{2 (q-1)^{2} q^{2j}}\\
C_m &= \frac{Q q^{2}(q-2)^{2} q^{4j} + Q q^{2}(q-2) q^{2j}\left(1+q^{n}\right) + q^{3} q^{n}}{2 (q-1)^{2} q^{2j}}, \qquad K_m = -\frac{1}{4}, \qquad D_m = -\frac{1}{2} \, q + 1
\end{align*}
\smallskip\noindent $m=3$, $m=9j+3$,  $j=0$:
\begin{align*}
A_m &= -\frac{Q q(q+2) + (q-2)\left(q^{2}+2\right) q^{n}}{8 q (q-1)}, \qquad B_m = \frac{Q\left(q^{2}-2\right) + \left(q^{3}-2q^{2}+4q-2\right) q^{n}}{2 (q - 1)}\\
C_m &= -\frac{q^{2}\left(Q(q-2) + \left(q^{2}-2q+2\right) q^{n}\right)}{2 (q - 1)}, \qquad K_m = \frac{1}{4}, \qquad D_m = -\frac{1}{2} \, q + 1
\end{align*}
\smallskip\noindent $m=9j+3$, $j≥1$, $n-2j≥3$:
\begin{align*}
A_m &= 2 \, q^{-2 \, j + n - 3}, \qquad B_m = -\frac{Q q\, q^{2j}\left(1+q^{n}\right) + 2(q-2) q^{n}}{q(q-1) q^{2j}}\\
C_m &= \frac{Q q^{2}(q-2) q^{2j}\left(1+q^{n}\right) + Q q^{4} q^{4j} + q(q-2)^{2} q^{n}}{2 (q-1)^{2} q^{2j}}, \qquad K_m = 2 \, q^{-2 \, j + n - 3}\\
D_m &= -\frac{Q q^{2 \, j + 1} - q^{n}}{{\left(q^{2} - q\right)} q^{2 \, j}}
\end{align*}
\smallskip\noindent $m=4$, $m=9j+4$,  $j=0$:
\begin{align*}
A_m &= 2 \, q^{n - 3}, \qquad B_m = -\frac{{\left(q^{3} - q^{2} + 3 \, q - 4\right)} q^{n} + Q q}{q^{2} - q}\\
C_m &= \frac{q\left(Q q(q+2) + (q-2)\left(q^{2}+2\right) q^{n}\right)}{2 (q - 1)}, \qquad K_m = 2 \, q^{n - 3}, \qquad D_m = -\frac{Q q - q^{n}}{q^{2} - q}
\end{align*}
\smallskip\noindent $m=9j+4$, $j≥1$, $n-2j≥2$:
\begin{align*}
A_m &= \frac{Q q^{2}(q-1) q^{4j} + Q q\, q^{2j}\left(1-q^{n}\right) - (q-1) q^{n}}{q(q-1) q^{2j}}, \qquad B_m = -\frac{Q q\, q^{2j}\left(1-q^{n}\right) - 2(q-1) q^{n}}{q(q-1) q^{2j}}\\
C_m &= -q^{-2 \, j + n - 1}, \qquad K_m = -1, \qquad D_m = 1
\end{align*}
\smallskip\noindent $m=5$, $m=9j+5$,  $j=0$:
\begin{align*}
A_m &= \frac{Q^{2} q - {\left(q^{3} - q^{2} + 2 \, q - 1\right)} q^{n}}{q^{2} - q}, \qquad B_m = \frac{{\left(q^{3} - q^{2} + 3 \, q - 2\right)} q^{n} - Q q}{q^{2} - q}\\
C_m &= -q^{n - 1}, \qquad K_m = -1, \qquad D_m = q + 1
\end{align*}
\smallskip\noindent $m=9j+5$, $j≥1$, $n-2j≥0$:
\begin{align*}
A_m &= Q q^{2 \, j - 1}, \qquad B_m = -\frac{Q\left(2 q(q-1) q^{2j} - q^{n} + 1\right)}{q - 1}\\
C_m &= \frac{Q q^{3}(q-1) q^{4j} + Q q^{2} q^{2j}\left(1-q^{n}\right) - q(q-1) q^{n}}{(q - 1) q^{2j}}, \qquad K_m = q^{2 \, j - 1}, \qquad D_m = \frac{q^{2 \, j + 1} - 1}{q - 1}
\end{align*}
\smallskip\noindent $m=6$, $m=9j+6$,  $j=0$:
\begin{align*}
A_m &= \frac{Q q(q+1) - \left(q^{2}+1\right) q^{n}}{q - 1}, \qquad B_m = -\frac{Q\left(2 q^{3}+1\right) - \left(2 q^{3}-q^{2}+3q-1\right) q^{n}}{q - 1}\\
C_m &= \frac{q\left(Q^{2} q - \left(q^{3}-q^{2}+2q-1\right) q^{n}\right)}{q - 1}, \qquad K_m = -q, \qquad D_m = Q
\end{align*}
\smallskip\noindent $m=9j+6$, $j≥1$, $n-2j≥0$:
\begin{align*}
A_m &= -\frac{Q q^{3} q^{4j} - Q q\, q^{2j}\left(1+q^{n}\right) + q^{n}}{(q-1)^{2} q^{2j}}, \qquad B_m = \frac{Q\left(2 q^{2} q^{2j} - q^{n} - 1\right)}{q - 1}, \qquad C_m = -Q q^{2 \, j + 1}\\
K_m &= q, \qquad D_m = -q + 1
\end{align*}
\smallskip\noindent $m=9j+7$, $j≥0$, $n-2j≥4$:
\begin{align*}
A_m &= q^{-2 \, j + n - 4}, \qquad B_m = -\frac{Q q\, q^{2j}\left(1+q^{n}\right) - 2 q^{n}}{q(q-1) q^{2j}}\\
C_m &= \frac{Q q^{5} q^{4j} - Q q^{3} q^{2j}\left(1+q^{n}\right) + q^{2} q^{n}}{(q-1)^{2} q^{2j}}, \qquad K_m = q^{-2 \, j + n - 4}, \qquad D_m = -\frac{Q q^{2 \, j + 2} - q^{n}}{{\left(q^{3} - q^{2}\right)} q^{2 \, j}}
\end{align*}
\smallskip\noindent $m=9j+8$, $j≥0$, $n-2j≥3$:
\begin{align*}
A_m &= \frac{Q q^{4}(q-1) q^{4j} - Q q^{2} q^{2j}\left(1-q^{n}\right) - (q-1) q^{n}}{q^{2}(q-1) q^{2j}}\\
B_m &= -\frac{Q q^{2} q^{2j}\left(1-q^{n}\right) + 2(q-1) q^{n}}{q^{2}(q-1) q^{2j}}, \qquad C_m = -q^{-2 \, j + n - 2}, \qquad K_m = 1, \qquad D_m = 1
\end{align*}
\smallskip\noindent $m=9j+9$, $j≥0$, $n-2j≥1$:
\begin{align*}
A_m &= -Q q^{2 \, j}, \qquad B_m = \frac{Q\left(2 q^{2}(q-1) q^{2j} + q^{n} - 1\right)}{q - 1}\\
C_m &= -\frac{Q q^{4}(q-1) q^{4j} - Q q^{2} q^{2j}\left(1-q^{n}\right) - (q-1) q^{n}}{(q - 1) q^{2j}}, \qquad K_m = -q^{2 \, j}\\
D_m &= \frac{q^{2 \, j + 2} - 1}{q - 1}
\end{align*}

\section{Pass $2$}\label{sec:pass2}

This appendix lists the tuples of pass~$2$, indexed by $m=9j+r$
($r\in\{1,\dots,9\}$) and written in terms of $q^{2j}$ and $q^n$. Here $m$ is an index local to pass~$2$: where
pass~$2$ sits in the sequence, and at which $j$ its rows are read, is described in
\S\ref{ssec:tables}. The formulas hold for every $n≥4$, even or
odd. As in Appendix~\ref{sec:pass1}, each row is written with the range on
which its triple $(A,B,C)$ satisfies \eqref{eq:adm}. A bound $j≥\tfrac12$ means $j≥1$
when the rows are read at integer $j$ (even $n$).

\smallskip\noindent $m=9j+1$, $j≥1$, $n-2j≥0$:
\begin{align*}
A_m &= \frac{Q q\left(q^{2j} + (q^{2j}-1) q^{n}\right) - q^{4 j + 2}}{(q-1)^{2} q^{2 j}}, \qquad B_m = -\frac{Q(q^{n}+1) - 2 q^{2 j + 1}}{q - 1}, \qquad C_m = -q^{2 \, j}\\
K_m &= q, \qquad D_m = -q + 1
\end{align*}
\smallskip\noindent $m=9j+2$, $j≥\tfrac12$, $n-2j≥3$:
\begin{align*}
A_m &= Q q^{-2 \, j + n - 3}, \qquad B_m = -\frac{Q\left(q^{2j} + (q^{2j}-2) q^{n}\right)}{(q-1) q^{2 j}}\\
C_m &= -\frac{Q q^{3}\left(q^{2j} + (q^{2j}-1) q^{n}\right) - q^{4 j + 4}}{(q-1)^{2} q^{2 j}}, \qquad K_m = q^{-2 \, j + n - 3}, \qquad D_m = -\frac{q^{2 \, j + 1} - q^{n}}{{\left(q^{2} - q\right)} q^{2 \, j}}
\end{align*}
\smallskip\noindent $m=9j+3$, $j≥\tfrac12$, $n-2j≥2$:
\begin{align*}
A_m &= \frac{q^{2}(q-1) q^{4j} + Q\left((q^{2j+1}-q+1) q^{n} - q^{2j+1}\right)}{(q^{2}-q) q^{2 j}}\\
B_m &= -\frac{Q\left(q^{2j+1} - (q^{2j+1}-2q+2) q^{n}\right)}{(q^{2}-q) q^{2 j}}, \qquad C_m = -Q q^{-2 \, j + n - 1}, \qquad K_m = 1\\
D_m &= 1
\end{align*}
\smallskip\noindent $m=9j+4$, $j≥\tfrac12$, $n-2j≥0$:
\begin{align*}
A_m &= -q^{2 \, j - 1}, \qquad B_m = \frac{2 q(q-1) q^{2j} + Q(q^{n}-1)}{q - 1}\\
C_m &= -\frac{(q^{4}-q^{3}) q^{4j} - Q\left(q^{2j+2} + (q^{2}-q-q^{2j+2}) q^{n}\right)}{(q-1) q^{2 j}}, \qquad K_m = -q^{2 \, j - 1}\\
D_m &= -\frac{Q - q^{2 \, j + 1}}{q - 1}
\end{align*}
\smallskip\noindent $m=9j+5$, $j≥\tfrac12$, $n-2j≥0$:
\begin{align*}
A_m &= -\frac{q(q-2)^{2} q^{4j} + Q(q-2) q^{2j}(1+q^{n}) + Q q^{n}}{(q-1)^{2} q^{2 j}}, \qquad B_m = -\frac{2 q(q-2) q^{2j} + Q(q^{n}+1)}{q - 1}\\
C_m &= -q^{2 \, j + 1}, \qquad K_m = 2, \qquad D_m = -\frac{1}{2} \, q + 1
\end{align*}
\smallskip\noindent $m=9j+6$, $j≥0$, $n-2j≥2$:
\begin{align*}
A_m &= \frac{Q(q-2)\left(q^{2j+2} + (q^{2j+2}+q-2) q^{n}\right) + q^{4 j + 5}}{8 q^{2}(q-1)^{2} q^{2 j}}\\
B_m &= \frac{(q^{4}-2q^{3}) q^{4j} + Q(q^{2}-2q+2) q^{2j}(1+q^{n}) + Q(q-2) q^{n}}{2(q-1)^{2} q^{2 j}}\\
C_m &= \frac{(q^{5}-4q^{4}+4q^{3}) q^{4j} + Q q^{2}(q-2) q^{2j}(1+q^{n}) + Q q^{2} q^{n}}{2(q-1)^{2} q^{2 j}}, \qquad K_m = -\frac{1}{4}\\
D_m &= -\frac{1}{2} \, q + 1
\end{align*}
\smallskip\noindent $m=9j+7$, $j≥0$, $n-2j≥4$:
\begin{align*}
A_m &= 2 \, Q q^{-2 \, j + n - 4}, \qquad B_m = -\frac{Q\left(q^{2j+2} + (q^{2j+2}+2q-4) q^{n}\right)}{(q^{3}-q^{2}) q^{2 j}}\\
C_m &= \frac{Q(q-2)\left(q^{2j+2} + (q^{2j+2}+q-2) q^{n}\right) + q^{4 j + 5}}{2(q-1)^{2} q^{2 j}}, \qquad K_m = 2 \, q^{-2 \, j + n - 4}\\
D_m &= -\frac{q^{2 \, j + 2} - q^{n}}{{\left(q^{3} - q^{2}\right)} q^{2 \, j}}
\end{align*}
\smallskip\noindent $m=9j+8$, $j≥0$, $n-2j≥3$:
\begin{align*}
A_m &= \frac{(q^{5}-q^{4}) q^{4j} + Q\left(q^{2j+2} - (q^{2j+2}+q-1) q^{n}\right)}{(q^{3}-q^{2}) q^{2 j}}\\
B_m &= -\frac{Q\left(q^{2j+2} - (q^{2j+2}+2q-2) q^{n}\right)}{(q^{3}-q^{2}) q^{2 j}}, \qquad C_m = -Q q^{-2 \, j + n - 2}, \qquad K_m = -1\\
D_m &= 1
\end{align*}
\smallskip\noindent $m=9j+9$, $j≥0$, $n-2j≥1$:
\begin{align*}
A_m &= q^{2 \, j}, \qquad B_m = -\frac{2 q^{2}(q-1) q^{2j} - Q(q^{n}-1)}{q - 1}\\
C_m &= \frac{(q^{5}-q^{4}) q^{4j} + Q\left(q^{2j+2} - (q^{2j+2}+q-1) q^{n}\right)}{(q-1) q^{2 j}}, \qquad K_m = q^{2 \, j}, \qquad D_m = -\frac{Q - q^{2 \, j + 2}}{q - 1}
\end{align*}

\section{Bridge $1$}\label{sec:bridge}

This appendix lists the eight tuples $\mathrm{b}1,\dots,\mathrm{b}8$ of
bridge~$1$. They do not involve $j$, and the formulas hold for every $n≥4$, even or odd.

\smallskip\noindent $\mathrm{b}1$:
\begin{align*}
A_{\mathrm{b}1} &= -\frac{Q q (q + 1) - (q^{2} + 1) q^{n}}{q - 1}, \qquad B_{\mathrm{b}1} = -\frac{{\left(q^{3} - q^{2} + q - 2\right)} q^{n} + Q q}{q^{2} - q}, \qquad C_{\mathrm{b}1} = -q^{n - 2}\\
K_{\mathrm{b}1} &= q, \qquad D_{\mathrm{b}1} = Q
\end{align*}
\smallskip\noindent $\mathrm{b}2$:
\begin{align*}
A_{\mathrm{b}2} &= \frac{Q^{2} q - {\left(q^{3} - q^{2} + 2 \, q - 1\right)} q^{n}}{q^{2} - q}, \qquad B_{\mathrm{b}2} = -\frac{Q (2 q^{3} + 1) - (2 \, q^{3} - q^{2} + 3 \, q - 1) q^{n}}{q - 1}\\
C_{\mathrm{b}2} &= \frac{Q q^{3}(q + 1) - q^{2}(q^{2} + 1) q^{n}}{q - 1}, \qquad K_{\mathrm{b}2} = -1, \qquad D_{\mathrm{b}2} = q + 1
\end{align*}
\smallskip\noindent $\mathrm{b}3$:
\begin{align*}
A_{\mathrm{b}3} &= -q^{n - 3}, \qquad B_{\mathrm{b}3} = \frac{{\left(q^{3} - q^{2} + 3 \, q - 2\right)} q^{n} - Q q}{q^{2} - q}\\
C_{\mathrm{b}3} &= \frac{Q^{2} q^{2} - {\left(q^{4} - q^{3} + 2 \, q^{2} - q\right)} q^{n}}{q - 1}, \qquad K_{\mathrm{b}3} = -q^{n - 3}, \qquad D_{\mathrm{b}3} = -\frac{Q q - q^{n}}{q^{2} - q}
\end{align*}
\smallskip\noindent $\mathrm{b}4$:
\begin{align*}
A_{\mathrm{b}4} &= -\frac{Q q (q + 2) + (q - 2)(q^{2} + 2) q^{n}}{q^{2} - q}, \qquad B_{\mathrm{b}4} = -\frac{{\left(q^{3} - q^{2} + 3 \, q - 4\right)} q^{n} + Q q}{q^{2} - q}\\
C_{\mathrm{b}4} &= -q^{n - 1}, \qquad K_{\mathrm{b}4} = 2, \qquad D_{\mathrm{b}4} = -\frac{1}{2} \, q + 1
\end{align*}
\smallskip\noindent $\mathrm{b}5$:
\begin{align*}
A_{\mathrm{b}5} &= \frac{(q^{2} - 2 \, q + 2) q^{n} + Q (q - 2)}{8 \, {\left(q - 1\right)}}, \qquad B_{\mathrm{b}5} = \frac{Q (q^{2} - 2) + (q^{3} - 2 \, q^{2} + 4 \, q - 2) q^{n}}{2 \, {\left(q - 1\right)}}\\
C_{\mathrm{b}5} &= \frac{Q q^{2}(q + 2) + q (q - 2)(q^{2} + 2) q^{n}}{2 \, {\left(q - 1\right)}}, \qquad K_{\mathrm{b}5} = \frac{1}{4}, \qquad D_{\mathrm{b}5} = \frac{1}{2} \, q + 1
\end{align*}
\smallskip\noindent $\mathrm{b}6$:
\begin{align*}
A_{\mathrm{b}6} &= 2 \, q^{n - 2}, \qquad B_{\mathrm{b}6} = \frac{{\left(q^{2} - 3 \, q + 3\right)} q^{n} - Q}{q - 1}, \qquad C_{\mathrm{b}6} = -\frac{q^{2}\left(Q (q - 2) + (q^{2} - 2 q + 2) q^{n}\right)}{2 \, {\left(q - 1\right)}}\\
K_{\mathrm{b}6} &= 2 \, q^{n - 2}, \qquad D_{\mathrm{b}6} = -\frac{Q - q^{n}}{q - 1}
\end{align*}
\smallskip\noindent $\mathrm{b}7$:
\begin{align*}
A_{\mathrm{b}7} &= q^{n + 1}, \qquad B_{\mathrm{b}7} = -\frac{{\left(q^{2} - 3 \, q + 1\right)} q^{n} + Q}{q - 1}, \qquad C_{\mathrm{b}7} = -q^{n}, \qquad K_{\mathrm{b}7} = q^{n + 1}\\
D_{\mathrm{b}7} &= -\frac{{\left(q^{2} - 3 \, q + 1\right)} q^{n} + Q}{q - 1}
\end{align*}
\smallskip\noindent $\mathrm{b}8$:
\begin{align*}
A_{\mathrm{b}8} &= q^{n - 2}, \qquad B_{\mathrm{b}8} = -\frac{{\left(q^{2} - 3 \, q + 1\right)} q^{n} + Q}{q - 1}, \qquad C_{\mathrm{b}8} = -q^{n + 3}, \qquad K_{\mathrm{b}8} = q^{n - 2}\\
D_{\mathrm{b}8} &= -\frac{Q - q^{n}}{q - 1}
\end{align*}

\section{Bridge $2$}\label{sec:bridge2}

This appendix lists the five tuples $\mathrm{c}1,\dots,\mathrm{c}5$ of
bridge~$2$. They do not involve $j$, and the formulas hold for every $n≥4$, even or odd.

\smallskip\noindent $\mathrm{c}1$:
\begin{align*}
A_{\mathrm{c}1} &= \frac{Q q^{n} + (2q-1)(q-2)}{8 \, (q-1)}, \qquad B_{\mathrm{c}1} = \frac{2 \, q^{3} + Q q^{n + 1} - 3 \, q^{2} + 2 \, q - 2}{2 \, {\left(q - 1\right)}}\\
C_{\mathrm{c}1} &= \frac{q^{2}(2q^{2}-q+2) + Q(q-2)(q+2) q^{n}}{2 \, (q-1)}, \qquad K_{\mathrm{c}1} = -\frac{1}{4}, \qquad D_{\mathrm{c}1} = \frac{3}{2} \, q + 1
\end{align*}
\smallskip\noindent $\mathrm{c}2$:
\begin{align*}
A_{\mathrm{c}2} &= \frac{2\left(2 q(q^{2}-q-1) + Q(q+1) q^{n}\right)}{q(q-1)}, \qquad B_{\mathrm{c}2} = -\frac{4 q^{3} - 7 q^{2} - q + 1 + Q(2q+1) q^{n}}{q - 1}\\
C_{\mathrm{c}2} &= \frac{q^{2}\left((2q-1)(q-2) + Q q^{n}\right)}{2 \, (q-1)}, \qquad K_{\mathrm{c}2} = 4, \qquad D_{\mathrm{c}2} = -2 \, q + 1
\end{align*}
\smallskip\noindent $\mathrm{c}3$:
\begin{align*}
A_{\mathrm{c}3} &= \frac{1}{2} \, Q q^{n - 3}, \qquad B_{\mathrm{c}3} = \frac{Q q^{n} - Q q}{q^{2} - q}, \qquad C_{\mathrm{c}3} = -\frac{2 q^{2}(q^{2}-q-1) + Q q(q+1) q^{n}}{2 \, (q-1)}\\
K_{\mathrm{c}3} &= \frac{1}{2} \, q^{n - 3}, \qquad D_{\mathrm{c}3} = \frac{q^{n} - q}{q^{2} - q}
\end{align*}
\smallskip\noindent $\mathrm{c}4$:
\begin{align*}
A_{\mathrm{c}4} &= \frac{2 q(q^{2}-q-1) + Q(q+1) q^{n}}{4 \, (q^{2}-q)}, \qquad B_{\mathrm{c}4} = \frac{Q q^{n} - Q q}{q^{2} - q}, \qquad C_{\mathrm{c}4} = -Q q^{n - 1}\\
K_{\mathrm{c}4} &= \frac{1}{2}, \qquad D_{\mathrm{c}4} = -2 \, q + 1
\end{align*}
\smallskip\noindent $\mathrm{c}5$:
\begin{align*}
A_{\mathrm{c}5} &= -\frac{2\left(Q q^{n} + (2q-1)(q-2)\right)}{q - 1}, \qquad B_{\mathrm{c}5} = -\frac{4 q^{3} - 7 q^{2} - q + 1 + Q(2q+1) q^{n}}{q - 1}\\
C_{\mathrm{c}5} &= -\frac{2 q^{2}(q^{2}-q-1) + Q q(q+1) q^{n}}{2 \, (q-1)}, \qquad K_{\mathrm{c}5} = 4, \qquad D_{\mathrm{c}5} = \frac{3}{2} \, q + 1
\end{align*}

\section{Bridge $3$}\label{sec:bridge3}

This appendix lists the three tuples $\mathrm{d}1,\mathrm{d}2,\mathrm{d}3$ of
bridge~$3$. They do not involve $j$, and the formulas hold for every $n≥4$, even or odd.

\smallskip\noindent $\mathrm{d}1$:
\begin{align*}
A_{\mathrm{d}1} &= \frac{Q q^{n} - q}{q - 1}, \qquad B_{\mathrm{d}1} = -\frac{Q\bigl((q-2)q^n + q\bigr)}{q^2-q}, \qquad C_{\mathrm{d}1} = -Q q^{n - 2}, \qquad K_{\mathrm{d}1} = q\\
D_{\mathrm{d}1} &= Q
\end{align*}
\smallskip\noindent $\mathrm{d}2$:
\begin{align*}
A_{\mathrm{d}2} &= -Q q^{n - 1}, \qquad B_{\mathrm{d}2} = \frac{Q\bigl((2q-1)q^n - 1\bigr)}{q-1}, \qquad C_{\mathrm{d}2} = \frac{q^{3} - Q q^{n + 2}}{q - 1}, \qquad K_{\mathrm{d}2} = -q^{n - 1}\\
D_{\mathrm{d}2} &= \frac{{\left(2 \, q - 1\right)} q^{n} - 1}{q - 1}
\end{align*}
\smallskip\noindent $\mathrm{d}3$:
\begin{align*}
A_{\mathrm{d}3} &= -\frac{Q q^{n} - q}{q - 1}, \qquad B_{\mathrm{d}3} = \frac{Q\bigl((2q-1)q^n - 1\bigr)}{q-1}, \qquad C_{\mathrm{d}3} = -Q q^{n + 1}, \qquad K_{\mathrm{d}3} = -q\\
D_{\mathrm{d}3} &= Q
\end{align*}

\bigskip
\noindent\textbf{Acknowledgements.} We used Claude Opus~5.5 (Anthropic) to draft parts of §§3--4 and to write the SageMath certification programs; all results were checked by the authors.


\begin{thebibliography}{MGO25}

\bibitem[Bar12]{Barry12}
P.~\bgroup\sc Barry\egroup{}.
\newblock On sequences with {$\{-1,0,1\}$} {H}ankel transforms.
\newblock arXiv{\string:}1205.2565, 2012.

\bibitem[EO26]{EO}
P.~\bgroup\sc Etingof\egroup{} \andname{} V.~\bgroup\sc Ovsienko\egroup{}.
\newblock Coefficients of $q$-real numbers: their combinatorial meaning and
  growth.
\newblock arXiv{\string:}2608.06761, 2026.

\bibitem[Eti25]{Etingof}
P.~\bgroup\sc Etingof\egroup{}.
\newblock On $q$-real and $q$-complex numbers.
\newblock arXiv{\string:}2508.08440, 2025.

\bibitem[Han16]{Han16}
G.-N. \bgroup\sc Han\egroup{}.
\newblock Hankel continued fraction and its applications.
\newblock {\em Adv. Math.} {\bf 303} (2016), 295--321.

\bibitem[HP25]{HP}
G.-N. \bgroup\sc Han\egroup{} \andname{} E.~\bgroup\sc Pedon\egroup{}.
\newblock Continued fractions and Hankel determinants for $q$-metallic numbers.
\newblock arXiv{\string:}2502.05993, 2025.

\bibitem[MGO20]{MGO20}
S.~\bgroup\sc Morier-Genoud\egroup{} \andname{} V.~\bgroup\sc
  Ovsienko\egroup{}.
\newblock $q$-deformed rationals and $q$-continued fractions.
\newblock {\em Forum Math. Sigma} {\bf 8} (2020), Paper No. e13, 55 pages.

\bibitem[MGO22]{MGO22}
S.~\bgroup\sc Morier-Genoud\egroup{} \andname{} V.~\bgroup\sc
  Ovsienko\egroup{}.
\newblock On $q$-deformed real numbers.
\newblock {\em Exp. Math.} {\bf 31}, 2 (2022), 652--660.

\bibitem[MGO25]{MGO}
S.~\bgroup\sc Morier-Genoud\egroup{} \andname{} V.~\bgroup\sc
  Ovsienko\egroup{}.
\newblock $q$-deformed rationals and irrationals.
\newblock arXiv{\string:}2503.23834, 2025.

\bibitem[OP25]{OP25}
V.~\bgroup\sc Ovsienko\egroup{} \andname{} E.~\bgroup\sc Pedon\egroup{}.
\newblock Continued fractions for {$q$}-deformed real numbers,
  {$\{-1,0,1\}$}-{H}ankel determinants, and {S}omos-{G}ale-{R}obinson
  sequences.
\newblock {\em Adv. in Appl. Math.} {\bf 162} (2025), Paper No. 102788, 37
  pages.

\bibitem[Ped26]{Pedon}
E.~\bgroup\sc Pedon\egroup{}.
\newblock Analytical properties of $q$-metallic numbers.
\newblock arXiv{\string:}2604.19898, 2026.

\bibitem[SX08]{SX08}
R.~A. \bgroup\sc Sulanke\egroup{} \andname{} G.~\bgroup\sc Xin\egroup{}.
\newblock Hankel determinants for some common lattice paths.
\newblock {\em Adv. in Appl. Math.} {\bf 40} (2008), 149--167.

\end{thebibliography}
\end{document}